\documentclass[a4paper, reqno, 11pt]{amsart}

\usepackage[mathscr]{eucal}
\usepackage{color}

\usepackage{latexsym, amsmath, amssymb, amsthm, mathtools}
\usepackage{cases, verbatim}
\usepackage[active]{srcltx}
\usepackage{esint}
\usepackage{hypernat}
\usepackage{ulem}

\usepackage[abbrev]{amsrefs}
\usepackage{enumerate}
\usepackage{mathrsfs, amsthm, mathtools}
\usepackage{color}

\theoremstyle{plain}
\newtheorem{thm}{Theorem}[section]

\newtheorem{defi}[thm]{Definition}

\newtheorem{lem}[thm]{Lemma}
\newtheorem{cor}[thm]{Corollary}

\newtheorem{prop}[thm]{Proposition}
\theoremstyle{remark}
\newtheorem{rem}[thm]{\bf{Remark}}
\numberwithin{equation}{section}

\numberwithin{equation}{section}

\numberwithin{equation}{section}

\newcommand{\ol}{\overline}
\newcommand{\tr}{\operatorname{trace}}

\def\R{{\mathbb R}}

\def\S{{\mathbf S}}

\def\O{{\mathcal O}}

\def\rn{{\mathbb R}^n}

\def\vep{\varepsilon}

\begin{document}
\title[A PDE approach to Borell--Brascamp--Lieb inequality]{A Parabolic PDE-based Approach\\ to Borell--Brascamp--Lieb inequality}
\author[K. Ishige]{Kazuhiro Ishige}
\address[K. Ishige]{Graduate School of Mathematical Sciences, 
University of Tokyo 
3-8-1 Komaba, Meguro-ku, Tokyo, 153-8914, Japan}
\email{ishige@ms.u-tokyo.ac.jp}

\author[Q. Liu]{Qing Liu} 
\address[Q. Liu]{Geometric Partial Differential Equations Unit, Okinawa Institute of Science and Technology Graduate University, 1919-1 Tancha, Onna-son, Kunigami-gun, 
Okinawa, 904-0495, Japan}
\email{qing.liu@oist.jp}

\author[P. Salani]{Paolo Salani}
\address[P. Salani]{Dipartimento di Matematica``U. Dini”, Universit\`{a} di Firenze, viale Morgagni 67/A, 50134 Firenze, Italy}
\email{paolo.salani@unifi.it}

\date{}

\begin{abstract}
In this paper, we provide a new PDE proof for the celebrated Borell--Brascamp--Lieb inequality. 
Our approach reveals a deep connection between the Borell--Brascamp--Lieb inequality and properties of diffusion equations of porous medium type pertaining to the large time asymptotics and preservation of a generalized concavity of the solutions.
We also recover the equality condition in the special case of the Pr\'ekopa--Leindler inequality by further exploiting known properties of the heat equation including the eventual log-concavity and backward uniqueness of solutions. 


\end{abstract}

\subjclass[2020]{39B62, 35E10, 35K05, 35D40}
\keywords{Borell--Brascamp--Lieb inequality, Pr\'ekopa--Leindler inequality, concavity properties, heat equation, porous medium equation, fast diffusion, viscosity solution}

\maketitle
\section{Introduction}
\subsection{Background}
Let  $\lambda\in(0,1)$ and $ f,g,h$ be nonegative measurable functions in $\rn$ such that
\begin{equation}\label{eq:1.1}
h((1-\lambda)y+\lambda z)\geq f(y)^{1-\lambda} g(z)^\lambda
\quad\text{for almost all $y$, $z\in\rn$}. 
\end{equation}
The {\it Pr\'ekopa-Leindler inequality} (also written as PL in brief below) states that 
\begin{equation}\label{eq:1.2}
\int_{\rn} h(x)\,dx\geq\left(\int_{\rn}  f (x)\,dx\right)^{1-\lambda}\left(\int_{\rn}  g(x) \,dx\right)^{\lambda}. 
\end{equation}
PL was proved by Pr\'ekopa~\cites{Pre1, Pre2} and Leindler~\cite{Lei} and it is considered an equivalent functional form of
the Brunn-Minkowski inequality, a crucial result in the theory of convex bodies.

PL is in fact a particular case of the {\it Borell--Brascamp--Lieb inequality}
(often shortened as BBL below). 
To state them, we introduce the power means of nonnegative numbers: for $a,b\geq 0$ and $\alpha\in[-\infty,+\infty]$, we set  
\begin{equation*}
\begin{array}{ll}
{M}_\alpha(a, b; \lambda)=0\quad\text{ for every }q\in\R\cup\{\pm\infty\}\quad&\text{if }ab=0, \\
\\
{M}_\alpha(a, b; \lambda) =\left\{
 \begin{array}{ll}
 \max\{a,b\}  \, \,\, \, &\,\,\alpha=+\infty,  \\
 \left[ (1-\lambda)a^\alpha + \lambda b^\alpha\right]^{\frac{1}{\alpha}} \, \, &\,\,\alpha\in\R\setminus\{0\},  \vspace{3pt}\\
a^{1-\lambda} b^\lambda \,\, \, \, &\,\,\alpha=0, \\
  \min\{a,b\} \, \, \, \, &\,\,\alpha=-\infty, 
                  \end{array}
\right.\,\quad&\text{if }ab>0. 
\end{array}
\end{equation*} 
We can now recall the Borell--Brascamp--Lieb inequality.
\begin{thm}
\label{thm:1.1}
Let $n\geq 1$, $\lambda\in(0,1)$ and $\alpha\in[-1/n,\infty]$.  
Assume that nonnegative integrable functions $f$, $g$, and $h$ in ${\mathbb R}^n$ satisfy 
\begin{equation}
\label{eq:1.3}
h((1-\lambda)y+\lambda z)\ge M_\alpha(f(y), g(z); \lambda)
\end{equation}
for almost all $y$, $z\in{\mathbb R}^n$. Then 
\begin{equation}
\label{eq:1.4}
\int_{\R^n}h(x)\,dx\ge
M_{\frac{\alpha}{1+n\alpha}}\left(\int_{\R^n} f(x)\,dx, \int_{\R^n} g(x)\,dx; \lambda\right).
\end{equation}
Here we interpret
$$
\frac{\alpha}{n\alpha+1}:=-\infty\quad\mbox{if}\quad\alpha=-\frac{1}{n},
\qquad
\frac{\alpha}{n\alpha+1}:=\frac{1}{n}\quad\mbox{if}\quad\alpha=\infty.
$$
\end{thm}
 We refer to the foundational works \cites{Bor1,BrLi, D, HM} about this inequality. To avoid triviality, throughout this paper we shall consider functions with positive $L^1$ norms. It is clear that PL corresponds to the case $\alpha=0$. PL and BBL can be generalized in several ways, including the case involving the mean of more than two functions (actually, of any finite number of functions), see for instance \cite{Ma} for a recent simple proof and more references. 
We also recall that in the one dimensional case, Dubuc proved a strengthened version of BBL in  \cite{Dubuc} (see also \cite{RS2}). 
Exactly as in the case of PL, BBL with any $\alpha\in[-1/n,\infty]$ is equivalent to the Brunn-Minkowski inequality and it is related to many other important geometric and analytic inequalities; see \cite{Gar} for a  comprehensive and insightful survey on this topic along with more related references. 

A delicate question is the rigidity of PL and BBL, that is, to determine whether and when equality can hold in \eqref{eq:1.4}. 
This question has been solved by Dubuc in \cite{Dubuc}, which proved that equality in \eqref{eq:1.4} holds if and only if 
$f$, $g$, $h$ 
and they all coincide, up to positive scalar multiplication and homotheties. 
We recall that a nonnegative function $v$ is called {\it $\alpha$-concave} in ${\mathbb R}^n$ if $v$ satisfies 
\[
v((1-\lambda)y+\lambda z)\ge M_\alpha(v(y),v(z);\lambda)
\quad\mbox{for all $y$, $z\in{\mathbb R}^n$ and $\lambda\in(0,1)$.}
\]
This is equivalent to saying that
  the positivity set $P_v:=\{x\in\rn\,:\,v(x)>0\}$ of $v$ is convex and 
  $$
  \left\{
  \begin{array}{ll}
  v\,\text{is a positive constant in }P_v&\text{if}\quad\alpha=+\infty;\\
  v^\alpha\,\,\text{ is concave in }P_v&\text{if}\quad\alpha>0;\\
  \log v\,\,\text{ is concave in }P_v&\text{if}\quad\alpha=0;\\
  v^\alpha\,\,\text{ is convex in }P_v&\text{if}\quad\alpha<0;\\
  \text{every superlevel set of }v\text{ is convex}&\text{if}\quad\alpha=-\infty. 
  \end{array}\right.
  $$
When $\alpha=0$, $v$ is also said {\it log-concave}; when $\alpha=-\infty$, $v$ is also said {\it quasi-concave}. 
Given a proper subset $\Omega$ of $\rn$, we say that $v:\Omega\to[0,+\infty)$ is $\alpha$-concave in $\Omega$ if the function $v^\alpha$ obtained by zero extension outside $\Omega$ is $\alpha$-concave in $\rn$. The result of Dubuc is explicitly stated in the following theorem; consult also \cite{BK} and \cite{R}.

\begin{thm}\label{thm:1.2}
In the same assumptions and notation of Theorem~{\rm\ref{thm:1.1}}, equality holds in~\eqref{eq:1.4} if and only if 
$f$, $g$, $h$ are $\alpha$-concave in ${\mathbb R}^n$ and there exist suitable $a$, $b$, $c$, $d\in(0,\infty)$ and $x_0,x_1\in\rn$ such that 
\begin{equation}\label{eq:1.5}
g(x)=a\,f(bx-x_0), \quad h(x)=c\,f(dx-x_1) \quad\text{for almost all }x\in\rn.
\end{equation}

\end{thm}

Once the rigidity is established, with precise equality conditions, it is natural to ask about the stability of an inequality. For BBL the question is the following: if the equality in \eqref{eq:1.4} nearly holds (in certain sense), then must the three functions $ f,g,h$ almost satisfy the conditions dictated by \cite{Dubuc}? In particular,  are they necessarily close (in some suitable sense) to be $\alpha$-concave and homothetic to each other? Notice that the two questions about $\alpha$-concavity and homothety are distinct.  This is an open and current subject of investigation and
as far as we know, there are very few results regarding the stability of PL or more in general of BBL, most of them assuming $\alpha$-concavity or however restricting to a special class of functions (see \cites{BB1,BB2,BD,BF,E,GS}). To the best of our knowledge, only \cites{BFR,RS1,RS2} address this question without restrictions on the involved functions, hence also investigating  the ``stability" of the $\alpha$-concavity property. Moreover, apart from \cite{RS2}, which is however limited to the strengthened one dimensional case, we are not aware of any papers treating the stability in the case $\alpha<0$. Notice that the stability of BBL is closely related to the same question for the Brunn-Minkowski inequality. The literature about the latter is certainly more extensive and often dated, but breakthroughs beyond the class of convex sets have only emerged more recently; 
see \cites{Ch1, Ch2, Ch3, FJ1,FJ2, FJ3} and the latest (and sharp) results \cite{HST,FHT}. 

In addition to the geometric flavor evoked by their equivalence with the Brunn-Minkowski inequality,  
PL and BBL have may interesting applications to PDEs and calculus of variations. In particular, PL enables to prove rather
easily that the log-concavity of the initial datum is preserved by the heat flow. Such a proof is accredited to Brascamp and Lieb \cite{BrLi}, but it seems to have been first explicitly written by  Deslauriers and Dubuc in~\cite{DeDu}. 
The primary objective of this paper is to show that the reverse implication holds as well, not only for PL but also for general BBL with any $\alpha\in(-1/n,0]$. 

Notice that the previously mentioned known proofs of BBL and PL 
adopt techniques from either convex analysis and geometry or optimal transportation. 
In contrast, our present work provides a PDE-based approach to BBL for $\alpha\leq 0$. It is based on a power concavity  property of solutions to suitable parabolic equations, which has been successfully developed earlier in slightly different contexts 
(see for instance \cites{IshLS, IshSaTa2} and references therein). We will demonstrate how these PDE arguments can also be adapted to prove PL and all BBL.  Our techniques 
for the case of PL will be briefly outlined in Section~\ref{sec:pl-sketch}. 
This PDE approach to PL however is not completely new: an analogous idea was introduced in \cite{BaCo}.
It was inspired by Borell, who gave an earlier similar proof of Ehrhard's inequality in  \cite{Bor2} and later improved in \cite{SvH} to also catch the equality case. We remark that BBL for a specific $\alpha$ implies BBL for all $\beta\geq\alpha$ by a classical homogeneity argument, which we will present in Appendix \ref{section:A} for the convenience of the reader. Consequently, our analysis for the case $\alpha<0$ not only generalizes the results in \cite{BaCo} but also improves this PDE method.  

Furthermore, we develop our PDE approach also for the delicate equality conditions in Theorem~\ref{thm:1.2}. 
Let us stress that, although our argument for equality conditions formally works for every $\alpha$, here we can give a complete proof only for the case $\alpha=0$, assuming also that $f, g, h$ are continuous functions with compact support; see Theorem \ref{thm:5.1}. Clearly, by the same homogeneity argument presented in Appendix \ref{section:A}, this yields the equality conditions for any $\alpha\geq 0$, but a completely sound PDE proof for $\alpha\in[-1/n,0)$ remains out of our reach in this paper, as we will explain in detail later.

In summary, our major contributions in this paper are new proofs, purely based on PDE methods, of Theorem~\ref{thm:1.1} for $\alpha\leq 0$ and of Theorem~\ref{thm:1.2} for continuous compactly supported functions in the case $\alpha=0$. 
We hope that our techniques can help to shed new light on the challenging question of stability.



\subsection{A sketch of the argument for PL ($\alpha=0$)}\label{sec:pl-sketch}
Let us first give a brief sketch of the PDE argument in the case of PL, i.e., $\alpha=0$. While the PDE proof of this inequality in \cite{BaCo} additionally employs a stochastic interpretation of linear diffusion equations, we only utilize PDE tools, applying the viscosity solution theory. 

For $\phi\in L^1(\rn)$, let us consider the problem
\begin{equation}\label{eq:1.6}
\left\{\begin{array}{ll}
\partial_t u=\Delta u\quad&\mbox{in}\quad{\mathbb R}^n\times(0,\infty),\\
u(\cdot,0)=\phi\quad&\mbox{in}\quad{\mathbb R}^n, 
\end{array}\right.
\end{equation}
and its solution given by
\[
u(x,t)=(e^{t\Delta_{{\mathbb R}^n}}\phi)(x):=(4\pi t)^{-\frac{n}{2}}\int_{{\mathbb R}^n}e^{-\frac{|x-y|^2}{4t}}\phi(y)\,dy. 
\]
Assume that $f$, $g$, $h$ are nonnegative functions in $L^1({\mathbb R}^n)$. 
Let $u_0$, $u_1$, $u_\lambda$ be the solutions with initial values $f, g, h$, respectively, that is, 
$$
u_0(x,t)=(e^{t\Delta_{{\mathbb R}^n}}f)(x), \quad u_1(x,t)=(e^{t\Delta_{{\mathbb R}^n}}g)(x), \quad u_\lambda(x,t)=(e^{t\Delta_{{\mathbb R}^n}}h)(x).
$$
Note that Lebesgue's dominated convergence theorem implies that 
\begin{equation}\label{eq:1.7}
\begin{split}
 & \lim_{t\to\infty}(4\pi t)^{\frac{n}{2}}u_0(x,t)=\int_{{\mathbb R}^n}f(y)\,dy, \quad
\lim_{t\to\infty}(4\pi t)^{\frac{n}{2}}u_1(x,t)=\int_{{\mathbb R}^n}g(y)\,dy, \\
 & \lim_{t\to\infty}(4\pi t)^{\frac{n}{2}}u_\lambda(x,t)=\int_{{\mathbb R}^n}h(y)\,dy,
 \quad\mbox{for every $x\in{\mathbb R}^n$}.
\end{split}
\end{equation}
We further set
$$
\ol{u}_\lambda(x,t):=\sup\left\{u_0(y,t)^{1-\lambda}u_1(z,t)^\lambda\,:\, (1-\lambda)y+\lambda z=x\right\},
$$
and notice that, for all $x\in\rn$, 
$$
\ol{u}_\lambda(x,0)=\sup\left\{ f(y)^{1-\lambda} g(z)^\lambda\,:\, (1-\lambda)y+\lambda z=x\right\}\leq h(x), 
$$
thanks to \eqref{eq:1.1}. Now the key step consists in showing that
{\it $\ol{u}_\lambda$ is a viscosity subsolution of~\eqref{eq:1.6}}, with initial datum $\phi=\ol{u}_\lambda(x,0)$. Then, by the {\it comparison principle} we have
$$
\ol{u}_\lambda\leq u_\lambda\quad\text{in }\rn\times[0,\infty), 
$$
whence
$$
(4\pi t)^{\frac{n}{2}}u_\lambda(x,t)\geq (4\pi t)^{\frac{n}{2}}\ol{u}_\lambda(x,t)\ge \left[(4\pi t)^{\frac{n}{2}}u_0(x,t)\right]^{1-\lambda}\left[(4\pi t)^{\frac{n}{2}}u_1(x,t)\right]^\lambda
$$
for all $x\in\rn$ and $t>0$. This implies that
$$
(4\pi t)^{\frac{n}{2}}u_\lambda(x,t)\ge \left[(4\pi t)^{\frac{n}{2}}u_0(x,t)\right]^{1-\lambda}\left[(4\pi t)^{-\frac{n}{2}}u_1(x,t)\right]^\lambda. 
$$
We then obtain PL by  letting $t\to+\infty$ for a fixed $x\in \R^n$ and using \eqref{eq:1.7}.

One technical issue in the argument above lies at possible insufficient regularity of initial data $f, g, h\in L^1(\R^n)$ to justify the application of comparison principle. This can be overcome by approximations via continuous functions. A similar difficulty also appears in \cite{BaCo} with a suitable approximation briefly mentioned. We will give full details about our approximation technique. 



We prove the general BBL with essentially the same argument as for PL, resting on suitable parabolic equations depending on $\alpha$. Hereafter we elaborate more details of our proof, including its delicate aspects. 

\subsection{Our PDE proof for BBL with $\alpha< 0$}
 Our argument for the case $\alpha<0$ is inspired by the power concavity of solutions to the nonlinear singular parabolic equation of porous medium type
\begin{equation}\label{eq:1.8}
\partial_t u={1\over m}\Delta (u^m) \quad \text{in $Q:=\R^n\times (0, \infty)$}
\end{equation}
with initial value 
\begin{equation}\label{eq:1.9}
u(\cdot, 0)=\phi \quad \text{in $\R^n$,}
\end{equation}
where $m$ is an exponent in the range $(-1, 1]$ and $\phi\in L^1(\R^n)\cap L^\infty(\R^n)$ is nonnegative and continuous. 
Here, $u^m/m$ should be understood as $\log u$ when $m=0$. 
Equation \eqref{eq:1.8} clearly reduces to the heat equation when $m=1$, otherwise the phenomenon underlying equation \eqref{eq:1.8} is usually called fast diffusion when $0<m<1$ and superfast diffusion when $m<0$.
The relation between the parameter $\alpha$ in Theorem~\ref{thm:1.1} and the exponent $m$ in \eqref{eq:1.8} is 
\begin{equation}\label{eq:1.10}
\alpha={m-1\over 2} \quad\text{ or equivalently } \quad m=2\alpha+1.
\end{equation}
In particular, the case of PL 
($\alpha=0$) corresponds to the heat equation ($m=1$) on the PDE side, as we have already seen.

In this work, in addition to the constraint $-1<m\leq 1$, we shall always focus on the so-called supercritical case, that is, 
\begin{equation}\label{eq:1.11}
n(m-1)+2>0. 
\end{equation}
 Well-posedness results for the Cauchy problem \eqref{eq:1.8}--\eqref{eq:1.9} in these circumstances can be found in \cites{ERV, HP, VaBook}. 
 
Notice that \eqref{eq:1.10} and \eqref{eq:1.11} imply $\alpha>-1/n$. We will get the limit case $\alpha=-1/n$ by approximation.
Moreover, in our analysis, we will often need to distinguish two different ranges of $m, n$: a) $0<m\leq 1$, $n(m-1)+2>0$ and b) $n=1, -1<m\leq 0$. This is due to different behaviors of the solutions, but the general structure of our proof for both cases is similar to the one for PL and consists of the following three steps. 
\begin{itemize}
\item[Step 1] 
We establish a concavity comparison between  solutions of \eqref{eq:1.8} with different initial data in even different domains, 
using the so-called spatial Minkowski convolution (see Sections~\ref{section:2} and \ref{section:3}). 
\item[Step 2] 
Based on the result obtained in Step 1, we use the large-time behavior of the Cauchy problem \eqref{eq:1.8}--\eqref{eq:1.9} to prove Theorem~\ref{thm:1.1} when $f, g, h\in C(\R^n)\cap L^1(\R^n)$ and satisfy also some other regularity assumptions (see Section~\ref{subsection:4.1}). 
\item[Step 3] 
We drop the continuity and the other extra assumptions on $f, g, h$ used in the result of Step 2 by appropriate approximations (see Section~\ref{subsection:4.2}). 
\end{itemize}
 \medskip
 
Let us now roughly describe our strategy of Step 1, which is crucial for our argument. 
Let $u_i\in C(\R^n\times [0, \infty))$ be a positive bounded classical solution of \eqref{eq:1.8} associated to the initial value $\phi_i$, with $i\in \{0, 1, \lambda\}$. For $\alpha=(m-1)/2$ and a fixed $\lambda\in (0, 1)$, let $\ol{u}_{\lambda}$ be the spatial $\alpha$-Minkowski convolution of $u_0, u_1$, defined by 
\begin{equation}\label{eq:1.12}
\begin{aligned}
  \ol{u}_{\lambda}(x, t):=\sup\left\{ 
M_{\alpha}\left(u_0(y, t), u_1(z, t);\lambda\right)\,
:\,  y, z\in \R^n, 
\,  x =(1-\lambda)y+  \lambda z
\right\} 
\end{aligned}
\end{equation}
for $(x, t)\in Q$.
Under suitable assumptions, we prove that $u_\lambda\geq \ol{u}_\lambda$ in $\R^n\times [0, \infty)$ if the initial values satisfy
\begin{equation}\label{eq:1.13}
\phi_\lambda((1-\lambda)y+\lambda z)\geq M_\alpha(\phi_0(y), \phi_1(z); \lambda)
\end{equation}
for all $y, z\in \R^n$. A more precise statement is as follows. 
\begin{thm}
\label{thm:1.3}
Let $n\geq 1$, $0<m\leq 1$ with $n(m-1)+2>0$ and $\lambda\in(0,1)$. Let $\alpha=(m-1)/2$.
Assume that $\phi_0$, $\phi_1$, $\phi_\lambda$ are nonnegative bounded functions in $L^1(\R^n)\cap C(\R^n)$ such that \eqref{eq:1.13} 
holds for all $y, z\in \R^n$. 
For each $i\in\{0, 1, \lambda\}$, 
let $u_i$ be a solution to the Cauchy problem \eqref{eq:1.8}--\eqref{eq:1.9} with $\phi=\phi_i$. 
Then 
\begin{equation}\label{eq:1.14}
u_{\lambda}((1-\lambda)y+\lambda z, t)
\ge  M_\alpha(u_{0})(y, t), u_{1}(z, t); \lambda)
\end{equation}
holds for all $y$, $z\in{\mathbb R}^n$ and $t>0$. 
\end{thm}

In order to obtain \eqref{eq:1.14}, we show that $\ol{u}_\lambda$ is a subsolution of \eqref{eq:1.8} and then we apply the comparison principle. This program can be carried out rigorously if we adopt the viscosity solutions theory. Moreover, instead of studying the Cauchy problem \eqref{eq:1.8}--\eqref{eq:1.9} directly, we first consider the following Cauchy--Dirichlet problem in bounded smooth domains:
\begin{numcases}{}
  \partial_tu={1\over m}\Delta (u^m)&\text{in $\Omega\times (0, \infty)$,}\label{eq:1.15}\\
    u=0& \text{on $\partial \Omega\times (0, \infty)$,}\label{eq:1.16}\\
  u(\cdot,0)=\phi &\text{in $\Omega$.}\label{eq:1.17}
  \end{numcases}
Considering the solution $u_i$ associated to a smooth domain $\Omega_i\subset \R^n$ and positive initial datum $\phi_i\in C(\ol{\Omega_i})$ for $i=0, 1$, 
we show that the spatial Minkowski convolution $\ol{u}_\lambda$ is a subsolution of $u_0$, $u_1$ (extended by $0$ outside of $\Omega_0$ and $\Omega_1$, respectively) is a subsolution of the same problem in the domain $(1-\lambda)\Omega_0+\lambda \Omega_1$ with initial data $\phi_\lambda$, satisfying \eqref{eq:1.13}.  The setting of Cauchy--Dirichlet problem facilitates our application of the standard comparison principle to get \eqref{eq:1.14} for all $y\in \Omega_0$, $z\in \Omega_1$ and $t>0$. Once the result is proved for bounded domains, the inequality for all $y, z\in \R^n$ is then obtained via approximation. See more details in Section~\ref{section:3}. 

It is worth remarking that, although we impose positive initial data for the Cauchy--Dirichlet problem, 
we still need to consider nonnegative solutions rather than positive solutions. 
It is indeed well known that when $0<m<1$, there is an extinction time $T>0$, that is, 
the solution $u>0$ in $\Omega\times (0, T)$ 
and $u\equiv 0$ in $\Omega\times [T, \infty)$, see e.g., \cites{DiDi, Sab1, Sab2}. In Section~\ref{section:2}, we provide a more general version of the subsolution property 
that applies to viscosity solutions of fully nonlinear equations that are positive up to the time of extinction; see more details in Theorem~\ref{thm:2.2}. 


In the case $n=1$ and $-1<m<0$, we obtain a counterpart of Theorem~\ref{thm:1.3} under additional regularity assumptions on $\phi_i$, 
stated more precisely in Theorem~\ref{thm:3.8}. The changes are due to different properties of the superfast diffusion: 
in order to ensure that the supremum in \eqref{eq:1.12} is indeed a maximum, which is attained at some $y, z \in \R$, 
we utilize a certain decay of solutions at the space infinity from  \cite{ERV}. 
The borderline case $m=0$ can be handled via approximation as $m\to 0-$. 

In Step 2 we combine the comparison results of the previous step with the large-time behavior of solutions of \eqref{eq:1.8}--\eqref{eq:1.9} to deduce the desired relation between $L^1$ norms of the initial values. The limits in \eqref{eq:1.7} are generalized to the other cases (under an appropriate decay condition on a solution $u$ at space infinity) as follows:
%
\begin{equation}\label{eq:1.18}
\lim_{t\to\infty}t^{n\over d} u(0, t)=C \|u(\cdot, 0)\|_{L^1(\R^n)}^{2\over d}
\end{equation}
for some $C>0$ depending only on $m$, $n$. Here $d:=n(m-1)+2=2(n\alpha+1)>0$.
This result was 
first established in \cite{FKa} under the condition $m>1$. 
See \cite{VaBook}*{Theorem~18.34} for the case of $0<m<1$, $n\geq1$, and \cite{ERV}*{Theorem~5} for the case of $-1<m\leq 0$, $n=1$. 
Now taking $\phi_0=f$, $\phi_1=g$ and $\phi_\lambda=h$, we can use \eqref{eq:1.18} to get, at least formally, that
\[
\|h\|_{L^1(\R^n)}^{2\over d}\geq M_\alpha\left(\|f\|_{L^1(\R^n)}^{2\over d}, \|g\|_{L^1(\R^n)}^{2\over d}; \lambda\right),
\]
which implies the desired BBL. 




The last step  consists in dropping the additional regularity assumptions by implementing an appropriate regularization. This is necessary even in the case of $\alpha=0$, as pointed out previously. In Section \ref{subsection:4.2}, with full details, we carefully construct smooth positive approximations of $\phi_i\in L^1(\R^n)$ that still satisfies \eqref{eq:1.13}. 

\subsection{Our PDE proof for equality condition of PL}
After proving BBL, 
in Section \ref{section:5} we further develop our PDE approach to investigate Theorem~\ref{thm:1.2}. 
By considering the solution $u_i$ of \eqref{eq:1.8} ($i\in \{0, 1, \lambda\}$) 
with continuous initial data $f$, $g$, $h$, 
respectively, 
we can use the equality condition
\begin{equation}\label{eq:1.19} 
\int_{\R^n}h(x)\,dx=
M_{\frac{\alpha}{1+n\alpha}}\left(\int_{\R^n} f(x)\,dx, \int_{\R^n} g(x)\,dx; \lambda\right)
\end{equation}
to deduce that 
$u_\lambda=\ol{u}_\lambda$ in $\R^n\times (0, \infty)$, where $\ol{u}_\lambda$ still stands for the Minkowski convolution as in \eqref{eq:1.12}, and this forces such strong relations between the Hessian matrices and the gradients of $u_0$, $u_1$ and $u_\lambda$ to yield the desired equality conditions for $f,g$ and $h$.

Under the additional assumption that $f, g$ and $h$ are continuous compactly supported functions, we can exploit the eventual (strong) log-concavity of solutions of the heat equation established in \cite{LV1} (see also \cites{Ish-pre23, IshSaTa1}), 
finding $t_\ast>0$ such that $\log u_i(\cdot, t_\ast)$ for $i=0, 1, \lambda$ are strongly concave in~$\R^n$. 
Through the concave conjugates of $\log u_i$, we then build the connection 
\begin{equation}\label{eq:1.20}
u_0(x, t_\ast)=k u_1(x+\eta, t_\ast)=k^\lambda u_\lambda (x+\lambda \eta, t_\ast)\quad \text{for all $x\in \R^n$}
\end{equation}
with some $k>0$ and $\eta\in \R^n$. An application of the backward uniqueness of the heat equation (see for example \cite{EKPV})  completes this last step and yields the log-concavity of $u_i$ along with the equality condition \eqref{eq:1.5} with $\alpha=0$ for PL. 

The general case with $\alpha\in [-1/n, 0$) is not entirely clear for the moment. We can repeat the same process as in the case of $\alpha=0$ to observe that $\ol{u}_\lambda=u_\lambda$, but we face new challenges in rigorously completing the rest of the  arguments with the fast diffusion equation \eqref{eq:1.8}. 
While a result on eventual power concavity for fast diffusions is available in \cite{LV1}*{Theorem~6.1} under certain decay conditions on the initial value, to the best of our knowledge, its backward uniqueness remains open. 
Moreover, even in the case $\alpha=0$, it is not clear if we can drop the assumption of the compact support of $\phi_i$, which is needed in our proof to apply the known results on eventual log-concavity of solutions. Appropriate approximations of $f, g, h$ with compactly supported functions that preserve both \eqref{eq:1.3} and \eqref{eq:1.19} may serve the purpose. 
It would be interesting to explore all these questions to establish a thorough PDE approach to BBL 
including its equality condition.

Let us quickly emphasize that we could generalize our arguments for the case involving the mean of more than two functions, as we do in Section \ref{section:2} for the basic technical results, but we omit the details here and leave them to the reader.

We conclude the introduction with a brief remark on topics related to optimal transport. 
Notably, the papers~\cites{Mc1, Mc2, Tr} have made significant contributions of establishing proofs of BBL 
and PL 
using mass transportation. See also \cite{CMSc} for a generalized approach on Riemannian manifolds. As highlighted in \cite{Mc2}, the functional inequalities can be understood as displacement convexity of internal energy of an interacting gas model; related discussions can also be found in \cite{Gar}. Interestingly, using the tool of gradient flows, 
the paper~\cite{Ot} elucidates the relation between such types of energy and the large-time behavior of porous medium equation. As the latter plays an important role in this work, our proof therefore can be viewed as a direct PDE expression of the optimal transport interpretation for the inequalities, mirroring the connection found in \cite{Ot}.

\section*{Acknowledgments}
The authors would like to thank Goro Akagi, Jenn-Nan Wang, Masahiro Yamamoto, Eiji Yanagida for valuable discussions and helpful references. 

K.I. was supported in part by JSPS KAKENHI Grant No.~19H05599.
Q.L. was supported in part by JSPS KAKENHI Grants No.~19K03574, No.~22K03396.
P.S. is supported by the project "Geometric-Analytic Methods for PDEs and Applications (GAMPA)", funded by European Union within the Next Generation EU PRIN 2022 program (D.D. 104 - 02/02/2022 Ministero dell’Universit\`a e della Ricerca).
%
\section{Spatial Minkowski convolution}\label{section:2}
Our main result on spatial Minkowski convolution in this section is applicable to viscosity solutions of a general class of nonlinear parabolic equations. See also \cite{IshLS} for related results and applications about the spatiotemporal Minkowski convolution.

First of all, it is useful to generalize the notion of power means of two numbers to powers means a given set of nonnegative numbers. 
Let $k=2,3,\dots$ and
\[
\Lambda_k:=\left\{(\lambda_1,\dots,\lambda_k)\in[0,1]^k\,:\,\sum_{i=1}^k\lambda_i=1\right\}.
\]
For any $\alpha \in [-\infty, +\infty]$, $a=(a_1, \ldots, a_k)\in [0, \infty)^k$ and $\lambda=(\lambda_1, \ldots, \lambda_k)\in \Lambda_k$, 
we define 
\[
M_\alpha(a; \lambda)=\begin{cases}
(\lambda_1 a_1^\alpha+\cdots +\lambda_k a_k^\alpha)^{1\over \alpha}  & \mbox{for}\quad \alpha\not\in\{0, \pm \infty\},\\
\prod_{i=1}^ka_i^{\lambda_i}& \mbox{for}\quad\alpha=0, \\
\min\{a_1, \ldots, a_k\} & \mbox{for}\quad\alpha=-\infty, \\
\max\{a_1, \ldots, a_k\} & \mbox{for}\quad\alpha=+\infty,
\end{cases}
\quad\mbox{if}\quad \prod_{i=1}^ka_i> 0\,
\]
and $M_\alpha(a; \lambda)=0$ otherwise.
In the case $k=2$, for any $\lambda\in[0,1]$,  
we use $M_\alpha(a; \lambda)$ to denote $M_\alpha(a; (1-\lambda, \lambda))$ for $a=(a_1, a_2)\in[0,\infty)^2$ and $\alpha\in [-\infty, +\infty]$.

Let us now briefly recall 
preliminaries about viscosity solutions of fully nonlinear parabolic equations
\begin{equation}\label{eq:2.1}
\partial_tu+F(x,t,u,\nabla u,\nabla^2 u)=0\quad \text{in $\O$,}
\end{equation}
where $\O$ is a given domain in $\R^n\times (0, \infty)$. For our purpose in this paper, we only consider positive solutions and take $F\in C(\O\times (0, \infty) \times{\mathbb R}^n\times\S^n)$ to be an elliptic operator, that is,
\[
F(x, t, r, \eta, X)\leq F(x, t, r, \eta, Y)
\]
holds for any $(x, t)\in \O$, $r>0$, $\eta\in \R^n$ and $X, Y\in \S^n$ satisfying $X\geq Y$. 
Here $\S^n$ denotes the set of all $n\times n$ symmetric matrices. Next, we recall the definition of viscosity solutions. We denote by $USC(\O)$ and $LSC(\O)$ the sets of upper and lower semicontinuous functions in $\O$, respectively.

\begin{defi}\label{thm:2.1}
A locally bounded function $u\in USC(\O)$ {\rm ({\it resp.~$u\in LSC(\O)$})} is a viscosity subsolution 
{\rm ({\it resp.~supersolution})} of \eqref{eq:2.1} if whenever there exist $\varphi\in C^2(\O)$ and $(x_\ast, t_\ast)\in \O$ such that $u-\varphi$ attains a maximum {\rm ({\it resp.~minimum})} in $\O$ at $(x_\ast, t_\ast)$, we have 
\[
\partial_t \varphi(x_\ast, t_\ast)+F(x_\ast, t_\ast, u(x_\ast, t_\ast), \nabla \varphi(x_\ast, t_\ast), \nabla^2 \varphi(x_\ast, t_\ast))\leq 0 \quad\text{{\rm ({\it resp.~$\geq 0$})}}.
\]
A function $u\in C(\O)$ is called a viscosity solution if it is both a viscosity subsolution and a viscosity supersolution. 
\end{defi}

\begin{rem}\label{rmk jets}
In the definition above, the global extremum in $\O$ can be replaced by a local extremum. Also, one can equivalently define viscosity solutions through the notion of semijets $P^{2, \pm} u$ and their closure $\ol{P}^{2, \pm} u$. Recall that, for any fixed $(x, t)\in \O$ and $u\in USC(\O)$, $P^{2, +} u(x, t)$ denotes the set of all $(\tau, \xi, X)\in \R\times \R^n\times \S^n$ satisfying 
\[
\begin{aligned}
 u(y, s)\leq u(x, t)+ \tau (s-t)+\langle \xi, y-x\rangle+ {1\over 2}\langle X(y-x), (y-x)\rangle + o(|t-s|+|x-y|^2),
\end{aligned}
\]
while for $u\in LSC(\O)$ we have $P^{2, -}u(x, t)=-P^{2, +}(-u)(x, t)$. Also, we say that $(\tau, \xi, X)\in \ol{P}^{2, +}u(x, t)$ (respectively, $(\tau, \xi, X)\in\ol{P}^{2, -}u(x, t)$) if there exist $(x_j, t_j)\in \O$ and $(\tau_j, \xi_j, X_j)\in P^{2, +}u(x_j, t_j)$ (respectively, $(\tau_j, \xi_j, X_j)\in P^{2, -}u(x_j, t_j)$) such that 
\[
(x_j, t_j)\to (x, t), \quad u(x_j, t_j)\to u(x, t),\quad \text{and} \quad (\tau_j, \xi_j, X_j)\to (\tau, \xi, X)
\]
as $j\to \infty$. We refer to \cite{CIL} for details. 
\end{rem}

Let $k\geq 2$ and fix $\lambda\in \Lambda_k$.
 For any $i\in\{1,\dots,k,\lambda\}$, 
let 
$
F_i\in C(\R^n\times (0, \infty)\times (0, \infty)\times{\mathbb R}^n\times\S^n)
$
be an elliptic operator.  Let $\{\Omega_i\}_{i=1}^k$  be bounded smooth domains in $\R^n$. 
We denote by $\Omega_\lambda$ the Minkowski combination (with coefficient $\lambda$) of $\{\Omega_i\}_{i=1}^k$, that is,
$$
\Omega_\lambda:=\sum_{i=1}^k\lambda_i\Omega_i
=\left\{\sum_{i=1}^k \lambda_i x_i\,:\,x_i\in\Omega_i,\,\,i=1,\dots,k\right\}.
$$
Let $T\in (0, \infty]$ and set $D_{i, T}:=\Omega_i\times(0, T)$ for $i\in\{1,\dots,k,\lambda\}$. For $i=1, \ldots, k$, assume that for any $i=1, \ldots, k$, $u_i\in 
USC({Q})$ 
is a bounded nonnegative viscosity subsolution of
\begin{equation}\label{eq:2.2}
  \partial_tu+F_i(x, t, u, \nabla u, \nabla^2 u)=0 \quad\text{in $D_{i, T}$}
\end{equation}
satisfying 
\begin{equation}\label{eq:2.3}
u=0\quad  \text{on $\partial \Omega_i\times (0, T)$.}
\end{equation}
Define, for $\alpha\in (-\infty, 0]$ 
and $(x,t)\in\overline{D_{\lambda, T}}$, 
\begin{equation}
\label{eq:2.4}
\begin{split}
 & U_{\alpha,\lambda}(x,t)\\
 & :=\sup\biggr\{ 
M_{\alpha}\left(u_1(y_1,t),\dots,u_k(y_k,t);\lambda\right)\,
:\, y_i\in \overline{\Omega_i}\, (i=1, \ldots, k),
\,\,\,\, x=\sum_{i=1}^k \lambda_i y_i\,\,
\biggr\}. \,
\end{split}
\end{equation} 
Then the following result holds. 
\begin{thm}
\label{thm:2.2}
Let $n\ge 1$, $k\ge 2$, and $\lambda\in\Lambda_k$. Fix $T\in (0, \infty]$. Assume that for any $i=1,\dots,k$, 
$u_i\in 
USC(\ol{D_{i, T}})$ is a positive viscosity subsolution of 
\eqref{eq:2.2} satisfying \eqref{eq:2.3}. 
Let $\alpha\in (-\infty,0]$. Assume that the following condition holds. 
\begin{itemize}
\item[{\rm (F)}]
For any $i\in \{1, 2, \ldots, k, \lambda\}$, $F_i\in C(D_{i, T}\times{\mathbb R}^n\times\S^n)$ is an elliptic operator such that for any fixed $\theta\in\R^n$ and $t\in (0, T)$,
 \begin{equation*}
  {\mathcal F}^{\theta,t}_{\lambda, \alpha} \left(\sum_{i=1}^k \lambda_i x_i,\sum_{i=1}^k \lambda_i r_i,\sum_{i=1}^k \lambda_i A_i\right)
  \le \sum_{i=1}^k \lambda_i {\mathcal F}^{\theta,t}_{i, \alpha}(x_i, r_i,A_i)
 \end{equation*}
 holds for all $(x_i, r_i, A_i)\in\Omega_i\times(0, \infty)\times\S^n$ with $i=1,\dots,k$,
 where 
 $$
 {\mathcal F}^{\theta,t}_{i, \alpha}(x,r,A)
 :=
 \left\{
 \begin{array}{ll}
 r^{1-\frac1\alpha}F_i\left(x,t,r^{\frac{1}{\alpha}},r^{\frac1\alpha-1}\theta,r^{\frac{1}{\alpha}-3}A\right)\quad & \mbox{if}\quad \alpha\not=0,\vspace{5pt}\\
 e^{-r}F_i\left(x,t, e^r, e^r \theta,e^r A\right) & \mbox{if}\quad \alpha=0,
 \end{array}
 \right.
 $$
 for $i\in\{1,\dots,k,\lambda\}$.
  \end{itemize}
Then $U_{\alpha, \lambda}$ defined by \eqref{eq:2.4} is a viscosity subsolution of \eqref{eq:2.2}--\eqref{eq:2.3} with $i=\lambda$. 
\end{thm}
\begin{proof}
Let $\alpha\in(-\infty, 0]$ and $\lambda=(\lambda_1,\dots,\lambda_k)\in\Lambda_k$. 
It is not difficult to prove that $U_{\alpha, \lambda}$ is upper semicontinuous and positive in $D_{i, T}$. 
Moreover, $U_{\alpha, \lambda}$ satisfies \eqref{eq:2.3} with $i=\lambda$. In fact, note that for each $(x, t)\in \Omega_\lambda\times (0, T)$ and any $x_i\in \Omega_i$ for $i=1, \ldots, k$ such that $x=\sum_{i=1}^k \lambda_i x_i$, when $x\to \partial \Omega_\lambda$, we have $x_i\to \partial \Omega_i$ and thus by \eqref{eq:2.3} $u_i(x_i, t)\to 0$. This implies that $U_\lambda(x, t)\to 0$ as $x\to \partial \Omega_\lambda$.

Let us verify the subsolution property of $U_{\alpha, \lambda}$. We first consider the case when $u_i>0$ in $D_{i, T}$ for all $i\in \{1, \ldots, k, \lambda\}$.
Suppose that there exist a function $\varphi\in C^2(D_{\lambda, T})$ and $(x_*,t_*) \in D_{\lambda, T}$ such that 
\[
\max_{D_\lambda}\,(U_{\alpha, \lambda}-\varphi) =U_{\alpha, \lambda}(x_\ast, t_\ast)-\varphi(x_\ast, t_\ast)=0.
\]
Here, without loss of generality we may assume that $U_{\alpha, \lambda}(x_\ast, t_\ast)=\varphi(x_\ast, t_\ast)$ by adding an appropriate constant to $\varphi$.  

Due to the blowup of the boundary value of $u_i^\alpha$ for all $i=1, \ldots, k$, we can find $x_i\in\Omega_i$ 
such that 
\begin{equation}\label{eq:2.6}
x_*=\displaystyle{\sum_{i=1}^k}\lambda_i x_i,\qquad 
U_{\alpha,\lambda}(x_*,t_*)
=M_{\alpha}\left(u_1(x_1,t_*),\dots,u_k(x_k,t_*);\lambda\right).
\end{equation}
It follows that
$$
(y_1, y_2, \ldots, y_k, t)\mapsto \sum_{i=1}^k \lambda_i u_i(y_i, t)^\alpha-\varphi\left(\sum_{i=1}^k y_i, t\right)^\alpha
$$
attains a minimum at $(x_1, x_2, \ldots, x_k, t_\ast)$  when $\alpha<0$ and 
\[
(y_1, y_2, \ldots, y_k, t)\mapsto \prod_{\substack{i=1, \ldots, k}}  u_i(y_i, t)^{\lambda_i} -\varphi\left(\sum_{i=1}^k y_i, t\right).
\]
 a maximum when $\alpha=0$. Set  
$
U_*:=U_{\alpha, \lambda}(x_*,t_*)$ and 
$u_{i, \ast}:=u_i(x_i,t_*)$
for $i\in\{1,\dots,k\}$.  With these notations, it is then obvious that 
\begin{equation}\label{eq:2.7}
\varphi(x_\ast, t_\ast)=U_\ast
=
\left\{
\begin{array}{ll}
\displaystyle{\left(\sum_{i=1}^k \lambda_i u_{i, \ast}^\alpha\right)^{\frac{1}{\alpha}}} & \mbox{when $\alpha<0$},\vspace{7pt}\\
\displaystyle{\prod_{i=1}^k u_{i, \ast}^{\lambda_k}} & \mbox{when $\alpha=0$}
\end{array}
\right.
\end{equation}

%
By the Crandall--Ishii lemma \cite{CIL}*{Theorem 8.3},  
for any $\varepsilon>0$ small, 
there exist 
$(\tau_i, \eta_i, Z_i)\in \overline{P}^{2, +} u_i(x_i, t_\ast)$ ($i=1, 2, \ldots k$) such that 
\begin{equation}\label{eq:2.8}
\sum_{i=1}^k \lambda_i u_{i, \ast}^{\alpha-1} 
\tau_i
=U_\ast^{\alpha-1} \partial_t \varphi(x_\ast, t_\ast),
\end{equation}
\begin{equation}\label{eq:2.9}
\theta:=u_{1, \ast}^{\alpha-1} \eta_1 
=\ldots =u_{k, \ast}^{\alpha-1}\eta_k 
=U_\ast^{\alpha-1}\nabla \varphi(x_\ast, t_\ast)
\end{equation}
and 
\begin{equation}\label{eq:2.10}
\begin{aligned}
&\begin{pmatrix}
\lambda_1 X_1 & & & \\
& \lambda_2 X_2 & & \\
& & \ddots & \\
& & & \lambda_{k} X_{k}
\end{pmatrix}\\
&\leq  
\begin{pmatrix}
\lambda_1^2 Y & \lambda_1\lambda_2 Y & \cdots & \lambda_1\lambda_{k} Y\\
\lambda_2 \lambda_1 Y & \lambda_2^2 Y & \cdots & \lambda_2 \lambda_{k} Y\\
\vdots & \vdots & \ddots & \vdots\\
\lambda_{k}\lambda_1 Y & \lambda_{k-1} \lambda_2 Y &  \cdots & \lambda_{k}^2 Y
\end{pmatrix}
+\varepsilon \begin{pmatrix}
\lambda_1^2 Y & \lambda_1\lambda_2 Y & \cdots & \lambda_1\lambda_{k} Y\\
\lambda_2 \lambda_1 Y & \lambda_2^2 Y & \cdots & \lambda_2 \lambda_{k} Y\\
\vdots & \vdots & \ddots & \vdots\\
\lambda_{k}\lambda_1 Y & \lambda_{k-1} \lambda_2 Y &  \cdots & \lambda_{k}^2 Y
\end{pmatrix}^2,
\end{aligned}
\end{equation}
where 
\[
\begin{aligned}
X_i&=u_{i, \ast}^{\alpha-1} 
Z_i+(\alpha-1) u_{i, \ast}^{\alpha-2} \eta_i\otimes \eta_i, \ i=1, 2, \ldots, k, \\
Y&=U_\ast^{\alpha-1} \nabla^2 \varphi(x_\ast, t_\ast) + (\alpha-1) U_\ast^{\alpha-2} \nabla \varphi(x_\ast, t_\ast) \otimes \nabla \varphi(x_\ast, t_\ast).
\end{aligned}
\] 
Here we recall from Remark \ref{rmk jets} the definition of $\overline{P}^{2, +} u$ for an upper semicontinuous function $u$.   
Take an arbitrary unit vector $w\in \R^n$. Multiplying \eqref{eq:2.10} by $(u_{1, \ast}^\alpha w, u_{2, \ast}^\alpha w, \ldots, u_{k, \ast}^\alpha w)$ from left and its transpose from right, by \eqref{eq:2.7} we get 
\[
\begin{aligned}
 & \sum_{i=1}^k \lambda_i u_{i, \ast}^{3\alpha-1} \langle  Z_i w, w\rangle+ (\alpha-1) \sum_{i=1}^k \lambda_i u_{i, \ast}^\alpha |\langle \theta, w\rangle|^2\\
& \leq 
\sum_{i=1}^k \left(\lambda_i u_{i, \ast}^\alpha\right)^2 \langle Y w, w\rangle +O(\varepsilon)\\
&\leq 
U_\ast^{3\alpha-1} \langle \nabla^2 \varphi(x_\ast, t_\ast) w, w\rangle+(\alpha-1) U_\ast^\alpha |\langle \theta, w\rangle|^2 +O(\varepsilon),
\end{aligned}
\]
which further yields  
\begin{equation}\label{eq:2.11}
U_\ast^{1-3\alpha} \sum_{i=1}^k \lambda_i u_{i, \ast}^{3\alpha-1}  Z_i
\leq 
\nabla^2 \varphi(x_\ast, t_\ast)+C\varepsilon I
\end{equation}
for some $C>0$. 
Also, in view of \eqref{eq:2.7}, \eqref{eq:2.8} and \eqref{eq:2.9}, we obtain 
\begin{equation}
\label{eq:2.12}
 \partial_t\varphi(x_*,t_*)
 =U_*^{1-\alpha}\sum_{i=1}^k\lambda_i\, u_{i, \ast}^{\alpha-1} \tau_i,\quad
 \nabla\varphi(x_*,t_*)
 =U_*^{1-\alpha}\sum_{i=1}^k\lambda_i\,u_{i, \ast}^{\alpha-1} \eta_i
 =U_*^{1-\alpha}\theta.
\end{equation}

Using \eqref{eq:2.11} and \eqref{eq:2.12} as well as the ellipticity of $F$, 
we deduce
\[
\begin{split}
& \partial_t\varphi(x_*,t_*)
+F_\lambda \left(x_*,t_*,U_*,\nabla\varphi(x_*,t_*),\nabla^2\varphi(x_*,t_*)+C\varepsilon I\right)\\
& \leq U_*^{1-\alpha}
\sum_{i=1}^k \lambda_i \frac{\tau_i}{u_{i, \ast}^{1-\alpha}}
+F_\lambda\biggr(x_*,t_*,U_*,U_*^{1-\alpha}\theta,U_*^{1-3\alpha}\sum_{i=1}^k \lambda_i u_{i, \ast}^{3\alpha-1} {Z_i} \biggr).\\
\end{split}
\]
Since $u_i$ is a viscosity subsolution of \eqref{eq:2.2}, we have
\begin{equation}
\label{eq:2.13}
\begin{split}
& \partial_t \varphi(x_*,t_*)
+F_\lambda \left(x_*,t_*,U_*,\nabla\varphi(x_*,t_*),\nabla^2\varphi(x_*,t_*)+C\varepsilon I\right)\\
& \leq -U_*^{1-\alpha}\sum_{i=1}^k 
\frac{\lambda_i}{u_{i, \ast}^{1-\alpha}} F_i(x_i,t_*, u_{i, \ast}, u_{i, \ast}^{1-\alpha}\theta, u_{i, \ast}^{1-3\alpha}A_{i})+\\
&\qquad\qquad\qquad
+F_\lambda\left(x_*,t_*,U_*,U_*^{1-\alpha}\theta,U_*^{1-3\alpha}\sum_{i=1}^k\lambda_i A_{i}\right),
\end{split}
\end{equation}
where 
$A_i=u_{i, \ast}^{3\alpha-1}{
Z_i}$ for $i=1,\dots,k.$

On the other hand, for $\alpha<0$ it follows from (F) that 
\begin{equation*}
\begin{split}
 & U_*^{1-\alpha}\sum_{i=1}^k
\frac{\lambda_i }{u_{i, \ast}^{1-\alpha}}F_i\left(x_i,t_*, u_{i, \ast}, u_{i, \ast}^{1-\alpha}\theta, u_{i, \ast}^{1-3\alpha}A_i\right)\\
 & =U_*^{1-\alpha}\sum_{i=1}^k
\lambda_i v_{i, \ast}^{1-\frac{1}{\alpha}}F_i\left(x_i,t_*, v_{i, \ast}^{\frac{1}{\alpha}}, v_{i, \ast}^{\frac{1}{\alpha}-1}\theta, 
\left(v_{i, \ast}\right)^{\frac{1}{\alpha}-3}A_{i}\right)\\
 & =U_*^{1-\alpha}\sum_{i=1}^k
\lambda_i{\mathcal F}^{\theta,t_*}_{i,\alpha}\left(x_i, v_{i, \ast}, A_i\right) \ge U_*^{1-\alpha}\,
 {\mathcal F}^{\theta,t_*}_{\lambda,\alpha}\left(\sum_{i=1}^k \lambda_ix_i,\sum_{i=1}^k \lambda_i v_{i, \ast},\sum_{i=1}^k \lambda_iA_i\right),
\end{split}
\end{equation*}
where $v_{i, \ast}:=(u_{i, \ast})^\alpha$. 
Since 
$$
\sum_{i=1}^k \lambda_i x_i=x_*,
\quad
\sum_{i=1}^k \lambda_i v_{i, \ast}
=M_\alpha\left(u_{1, \ast},\dots, u_{k, \ast};\lambda\right)^\alpha=U_*^\alpha,
$$
we have 
\begin{equation*}
\begin{split}
 & U_*^{1-\alpha}\sum_{i=1}^k
\frac{\lambda_i }{u_{i, \ast}^{1-\alpha}}F_i\left(x_i,t_*, u_{i, \ast}, u_{i, \ast}^{1-\alpha}\theta, u_{i, \ast}^{1-3\alpha}A_i\right)\\
 & \ge  U_*^{1-\alpha}
 {\mathcal F}^{\theta,t_*}_{\lambda, \alpha}\left(x_*,U_*^\alpha,\sum_{i=1}^k \lambda_iA_i\right) 
 =F_\lambda\left(x_*,t_*,U_*,U_*^{1-\alpha}\theta,U_*^{1-3\alpha}\sum_{i=1}^k \lambda_iA_{i}\right).
\end{split}
\end{equation*}
Similarly, if $\alpha=0$, 
setting $v_{i, \ast}:=\log u_{i, \ast}$, 
we have
$$
\sum_{i=1}^k \lambda_i v_{i, \ast}
=\log\left(\prod_{i=1}^ku_{i, \ast}^{\lambda_i}\right)=\log U_*,
$$
and we obtain
\begin{equation*}
\begin{split}
 & U_*\sum_{i=1}^k
 \frac{\lambda_i}{u_{i, \ast}}F_i\left(x_i,t_*, u_{i, \ast}, u_{i, \ast}\theta, u_{i, \ast}A_i\right)\\
 & =U_*\sum_{i=1}^k
\lambda_i e^{-v_{i, \ast}}F_i\left(x_i, t_*, e^{v_{i, \ast}}, e^{v_{i, \ast}}\theta, e^{v_{i, \ast}}A_i\right)
=U_*\sum_{i=1}^k
\lambda_i{\mathcal F}^{\theta,t_*}_{i,0}\left(x_i, v_{i, \ast}, A_i\right)\\
 & \ge U_*\,
 {\mathcal F}^{\theta,t_*}_{\lambda,0}\left(\sum_{i=1}^k \lambda_ix_i,\sum_{i=1}^k \lambda_i v_{i, \ast},\sum_{i=1}^k \lambda_iA_i\right)\\
 & =U_*
 {\mathcal F}^{\theta,t_*}_{\lambda, 0}\left(x_*,\log U_*,\sum_{i=1}^k \lambda_iA_i\right)
 =F_\lambda\left(x_*,t_*,U_*,U_*\theta,U_*\sum_{i=1}^k \lambda_iA_{i}\right).
\end{split}
\end{equation*}
These together with \eqref{eq:2.13} imply 
\[
\partial_t\varphi(x_*,t_*)
+F_\lambda(x_*,t_*,U_*,\nabla\varphi(x_*,t_*),\nabla^2\varphi(x_*,t_*)+C\varepsilon I)\leq 0.
\]
Letting $\varepsilon\to 0$, we see that 
\begin{equation}\label{eq:2.14}
\partial_t\varphi(x_*,t_*)
+F_\lambda(x_*,t_*,U_*,\nabla\varphi(x_*,t_*),\nabla^2\varphi(x_*,t_*))\leq 0, 
\end{equation}
and thus 
$U_{\alpha, \lambda}$ is a viscosity subsolution of \eqref{eq:2.2} with $i=\lambda$. The proof is complete now. 
\end{proof}

\begin{rem}
Since all classical subsolutions are viscosity subsolutions, our result above certainly holds if it is additionally assumed that $u_i\in C^{2; 1}(D_{i, T})\cap 
C(\overline{D_{i, T}})$ for all $i=1, 2, \ldots, k$. In fact, under this additional assumption, we can simplify our proof by replacing all $(\tau_i, \eta_i, Z_i)$ with $(\partial_t u_i(x_i, t_\ast), \nabla u_i(x_i, t_\ast), \nabla^2 u_i(x_i, t_\ast))$ and taking $\varepsilon=0$. 
\end{rem}




\begin{rem}\label{thm:2.4}
In Theorem~\ref{thm:2.2} we only focus on the Cauchy--Dirichlet problem, but the same argument also applies to the Cauchy problem with $\Omega_i=\R^n$ and $\alpha<0$ provided that the following decay condition holds for all $u_i$, $i=1, \ldots, k$: for any $t\in (0, T)$, there exists $C(t)>0$ such that 
\begin{equation}\label{eq:2.15}
\sup_{|x|\geq R}u_i(x, t)\leq C(t)R^{1\over \alpha}\quad \text{as $R\to \infty$.} 
\end{equation}
We obtain the same result in Theorem~\ref{thm:2.2} with all $\Omega_i$ replaced by $\R^n$: 
the Minkowski convolution $U_{\alpha,\lambda}$ is a subsolution in $\R^n\times (0, \infty)$. 
The decay condition \eqref{eq:2.15} is utilized to guarantee the existence of maximizers $x_i\in \R^n$. 
Moreover, in this case $U_{\alpha,\lambda}$ satisfies \eqref{eq:2.15} too. Indeed, for $\alpha<0$, since $u_i$ satisfies \eqref{eq:2.15}, we have 
\[
U_\lambda(x, t)^\alpha=\inf\left\{\sum_{i=1}^k \lambda_i u_i(x_i, t)^{\alpha}\,
:\,  y_i\in \R^n, 
\, x=\sum_{i=1}^k \lambda_i y_i
\right\}, 
\] 
which implies that, for any $x\in \R^n$ with $|x|$ large and any $y_i\in \R^n$ satisfying $x=\sum_{i=1}^k \lambda_i y_i$, 
\[
U_{\alpha,\lambda}(x, t)^\alpha\geq C^\alpha\sum_{i=1}^k \lambda_i |y_i|\geq C^\alpha |x|. 
\]
\end{rem}
\section{Power concavity preserving for diffusion equations
}\label{section:3}
In this section, we apply Theorem~\ref{thm:2.2} and Remark~\ref{thm:2.4} to the diffusion equation \eqref{eq:1.8} 
and prove Theorem~\ref{thm:1.3}, with several further consequences. 

\subsection{Fast diffusion in general dimensions}

We first consider the Cauchy--Dirichlet problem \eqref{eq:1.15}--\eqref{eq:1.17} in a bounded smooth domain $\Omega\subset \R^n$ with $0<m\le 1$. In general, finite time extinction occurs in the case $0<m<1$ even when the initial data are positive in the domain. We thus consider 
\begin{numcases}{}
\partial_t u-{1\over m}\Delta u^m=0 &\quad \text{in $\Omega\times (0, T)$}\label{positive pme}\\
u=0 &\quad \text{on $\partial \Omega\times (0, T)$} \label{positive bdry}
\end{numcases}
for some extinction time $T\in (0, \infty]$. We refer to \cite{BoF} for a recent survey on various properties of this Cauchy-Dirichlet problem. 

%
\begin{thm}
\label{thm:3.1}
Let $n\geq 1$ and $0<m\le 1$. Let $\Omega_0$, $\Omega_1 \subset \R^n$ be bounded smooth domains. 
Fix $\lambda\in (0, 1)$ and let $\Omega_\lambda$ be the Minkowski combination of $\Omega_0$ and $\Omega_1$, i.e., 
\begin{equation}\label{eq:3.1}
\Omega_\lambda=(1-\lambda) \Omega_0+\lambda \Omega_1. 
\end{equation}
Suppose that for each $i=0, 1$, $u_i \in USC(\overline{\Omega_i}\times [0, T_i])$ is a positive viscosity solution of the Cauchy--Dirichlet problem~\eqref{positive pme}--\eqref{positive bdry} with $\Omega=\Omega_i$ and 
extinction time $T=T_i\in (0, \infty]$, that is, $u_i>0$ in $\Omega_i\times [0, T_i)$ and $u_i=0$ in $\Omega_i\times [T_i, \infty)$. 
For $\alpha=(m-1)/2$, let $\ol{u}_{\lambda}$ be the spatial Minkowski convolution of $u_0$ and $u_1$, defined by 
\begin{equation}\label{eq:3.2}
\begin{aligned}
\ol{u}_{\lambda}(x,t):=\sup\big\{ 
M_{\alpha}\left(u_0(y,t), u_1(z, t); \lambda\right)
:\, y\in \overline{\Omega_0}, z\in \overline{\Omega_1}, 
\,  x=(1-\lambda)y+\lambda z
\big\}
\end{aligned}
\end{equation}
for $(x, t)\in \Omega_\lambda\times (0, \infty)$. Then $\ol{u}_{\lambda}$
is a viscosity subsolution of the Cauchy--Dirichlet problem~\eqref{positive pme}--\eqref{positive bdry} with $\Omega=\Omega_\lambda$ and extinction time $T=\ol{T}:=\min\{T_0, T_1\}$. 
\end{thm}

\begin{proof}
By definition of the power mean $M_\alpha$, it is clear that $\ol{u}_\lambda=0$ in $\ol{\Omega_\lambda}\times [\ol{T}, \infty)$ and $\ol{u}_\lambda>0$ in $\Omega_\lambda\times [0, \ol{T})$. 
Let us verify the subsolution property of $\ol{u}_\lambda$ in $\Omega_\lambda\times (0, \ol{T})$. 
Since $u_i$ is a positive solution of \eqref{positive pme}
with $\Omega=\Omega_i, T=T_i$ for $i=0, 1$, we can rewrite the equation in the form \eqref{eq:2.2}, where 
\begin{equation}\label{eq:3.3}
F_i(x,t,r,\theta, A):=-r^{m-1} \tr A-(m-1)r^{m-2}  |\theta|^2
\end{equation}
for $(x,t,r,\theta, A)\in \Omega_i\times (0, T_i)\times (0, \infty)\times \R^n \times \S^n$ and $i=0, 1$. 
For any fixed $0<\lambda<1$, we take $F_\lambda$ to be the same operator $F$ as in \eqref{eq:3.3}. Then, under the choice $\alpha=(m-1)/2$, we have
\[
\begin{aligned}
{\mathcal F}^{\theta,t}_{i, \alpha}(x, r, A) & =r^{1-\frac1\alpha}F\left(x,t,r^{\frac{1}{\alpha}},r^{\frac1\alpha-1}\theta,r^{\frac1\alpha-3}A\right)\\
&=-\tr A-(m-1)r |\theta|^2
\end{aligned}
\]
for all $(x, r, A)\in \Omega_i\times (0, \infty)\times\S^n$,
$(\theta,t)\in{\mathbb R}^n\times(0,T_i)$ and $i\in \{0, 1, \lambda\}$. 
Since 
${\mathcal F}^{\theta,t}_{i, \alpha}$ is linear in $r$ and $A$, we verify (F) immediately. 
Hence, by Theorem~\ref{thm:2.2},  we see that $\ol{u}_{\lambda}$ is a positive viscosity subsolution of \eqref{positive pme} with $i=\lambda$ 
satisfying $\ol{u}_\lambda=0$ on $\partial \Omega_\lambda\times (0, \ol{T})$.  
\end{proof}

We now present a comparison principle for the Cauchy--Dirichlet problem with $n\geq 1$ and $0<m\leq 1$. 
Related comparison results are studied in \cite{FPS}. Our argument below is slightly different. 
\begin{prop}\label{thm:3.2}
Assume that $n\geq 1$ and $0<m\leq 1$. Let $\Omega$ be a bounded domain in $\R^n$ and $D_T=\Omega\times (0, T)$ for some $T>0$.
Assume that $u\in USC(\overline{D_T})$ is a bounded nonnegative subsolution of \eqref{positive pme}--\eqref{positive bdry} satisfying $u>0$ in $D_T$ and 
$v\in C(\overline{D_T})\cap C^{2; 1}(D_T)$ is a classical supersolution of \eqref{positive pme} satisfying $v>0$ in $\ol{D_T}$.  
If $u(\cdot, 0)< v(\cdot, 0)$ in $\overline{\Omega}$, then $u\leq v$ in $\overline{\Omega}\times [0, T)$. 
\end{prop}
\begin{proof}
Suppose that $u-v$ attains a positive value in $D_T$. Then, there must exist $\hat{t}\in \Omega\times [0, T)$ such that 
\begin{align}
\label{comparison-time}
 & u\leq v\quad \text{in $\overline{\Omega}\times [0, \hat{t}]$},\\
\label{hypo-comparison}
 & \max_{x\in \overline{\Omega}} \left(u(x, s_j)-v(x, s_j)\right)>0 \quad \text{at a sequence $s_j\to \hat{t}+$.}
\end{align}
In view of the conditions that $u\leq v$ on $\overline{\Omega}\times \{0\}$ and $u-v\in USC(\ol{D_T})$, we have $\hat{t}>0$.

For $\tau\in(0, T-\hat{t})$ and $\vep>0$, let 
\[
\Psi_\vep(x, t):= u(x, t)-v(x, t)-{2\vep \over \hat{t}+\tau-t},\quad (x, t)\in \overline{\Omega}\times [0, \hat{t}+\tau).
\]
Note that 
\[
M(\vep):=\max_{\overline{\Omega}\times [0, \hat{t}+\tau)}\Psi_\vep
\]
is continuous with respect to $\vep\in (0, \infty)$. In view of the boundedness of $u, v$ in $\overline{D_T}$, we have $M(\vep)\to -\infty$ as $\vep\to \infty$. It also follows from \eqref{hypo-comparison} that $\lim_{\vep\to 0}M(\vep)>0$. Then, by the intermediate value theorem, we can choose $\vep=\vep(\tau)>0$ such that $M(\vep)=0$, 
which implies the existence of a maximizer $(\ol{x}, \ol{t})$ of $\Psi_\vep$ in $\overline{\Omega}\times [0, \hat{t}+\tau)$ satisfying 
\begin{equation}\label{psi-result0}
u(x, t)-v(x, t)-{2\vep \over \hat{t}+\tau-t}\leq u(\ol{x}, \ol{t})-v(\ol{x}, \ol{t})-{2\vep \over \hat{t}+\tau-\ol{t}}=0
\end{equation}
for all $(x, t)\in \overline{\Omega}\times [0, \hat{t}+\tau)$. In particular, we have 
\begin{align}
\label{psi-result1}
 & u(\ol{x}, \ol{t})-v(\ol{x}, \ol{t})-{\vep \over \hat{t}+\tau-\ol{t}}={\vep \over \hat{t}+\tau-\ol{t}}>0,\\
\label{psi-result2}
 & u(x, t)-v(x, t)-{\vep \over \hat{t}+\tau-t}\leq {\vep \over \hat{t}+\tau-t} \quad\text{for all $(x, t)\in \overline{\Omega}\times [0, \hat{t}+\tau)$.}
\end{align}
It is also clear from \eqref{psi-result0} that our choice of $\vep=\vep(\tau)$ satisfies $\vep\to 0$ as $\tau\to 0$.

For these $\tau$, $\vep>0$, we next define
\[
\Phi_\vep(x, t):= u(x, t)-v(x, t)-{\vep \over \hat{t}+\tau-t}, \quad (x, t)\in \overline{\Omega}\times [0, \hat{t}+\tau). 
\]
Suppose that $\Phi_\vep$ attains a maximum at a point $(x_\vep, t_\vep)\in \overline{\Omega}\times [0, \hat{t}+\tau)$. By \eqref{psi-result1}, we see that 
\begin{equation}\label{psi-result3}
\max_{\overline{\Omega}\times [0, \hat{t}+\tau)} \Phi_\vep=\Phi(x_\vep, t_\vep)>0. 
\end{equation}
On the other hand, it follows from \eqref{psi-result2} that 
\begin{equation}\label{psi-result4}
\max_{\overline{\Omega}\times [0, \hat{t}+\tau)} \Phi_\vep=\Phi(x_\vep, t_\vep)=u(x_\vep, t_\vep)-v(x_\vep, t_\vep)-{\vep \over \hat{t}+\tau-t_\vep}\leq {\vep \over \hat{t}+\tau-t_\vep}.
\end{equation}
By the relation \eqref{comparison-time}, we see that $t_\vep>\hat{t}$. Moreover, as $\tau$, $\vep(\tau)\to 0$ we can take a convergent sequence, still denoted by $(x_\vep, t_\vep)$ such that $(x_\vep, t_\vep)\to (\hat{x}, \hat{t})$ for some $\hat{x}\in \overline{\Omega}$. Noticing that along this subsequence we have
\[
u(\hat{x}, \hat{t})-v(\hat{x}, \hat{t})\geq \limsup_{\vep\to 0}\Phi_\vep(x_\vep, t_\vep)\geq 0, 
\]
we therefore obtain $\hat{x}\in \Omega$ in view of the boundary data of $u, v$. This implies that $x_\vep$ is bounded away from $\partial \Omega$ for all such $\tau$, $\vep>0$. 
By the definition of subsolutions, we deduce that 
\[
\partial_t \varphi(x_\vep, t_\vep)+ F(u(x_\vep, t_\vep), \nabla \varphi(x_\vep, t_\vep), \nabla^2 \varphi(x_\vep, t_\vep)) \leq 0,
\]
where $\varphi(x, t)=v(x, t)+{\vep\over \hat{t}+\tau-t}$ for $(x, t)\in D_T$ and 
\[
F(r, \theta, A)=-r^{m-1} \tr A-(m-1)r^{m-2}  |\theta|^2, \quad (r, \theta, A)\in \R\times \R^n\times \S^n. 
\]
This immediately yields
\begin{equation}\label{sigma sub}
\partial_t v(x_\vep, t_\vep)+ {\vep\over (\hat{t}+\tau-t_\vep)^2}+ F\left(u(x_\vep, t_\vep), \nabla v(x_\vep, t_\vep), \nabla^2 v(x_\vep, t_\vep)\right) \leq 0.
\end{equation}
Since for all such small $\tau, \vep>0$, $(x_\vep, t_\vep)$ stays in a compact set of $\Omega\times (0, T)$ and 
\[
{\vep\over \hat{t}+\tau-t_\vep}\leq u(x_\vep, t_\vep)-v(x_\vep, t_\vep)\leq {2\vep\over \hat{t}+\tau-t_\vep}
\]
due to \eqref{psi-result3} and \eqref{psi-result4}, we have local Lipschitz continuity of $r\mapsto F(r, \nabla v(x_\vep, t_\vep), \nabla^2 v(x_\vep, t_\vep))$ implying that 
\[
\begin{aligned}
&F(v(x_\vep, t_\vep), \nabla v(x_\vep, t_\vep), \nabla^2 v(x_\vep, t_\vep))\\
&\leq F(u(x_\vep, t_\vep), \nabla v(x_\vep, t_\vep), \nabla^2 v(x_\vep, t_\vep))+{C\vep\over \hat{t}+\tau-t_\vep}
\end{aligned}
\]
for some $C>0$ independent of $\tau, \vep>0$. It then follows from \eqref{sigma sub} that
\[
\partial_t v(x_\vep, t_\vep)+ F\left(v(x_\vep, t_\vep), \nabla v(x_\vep, t_\vep), \nabla^2 v(x_\vep, t_\vep)\right) 
\leq {C\vep\over \hat{t}+\tau-t_\vep}-{\vep\over (\hat{t}+\tau-t_\vep)^2}.
\]
Since $v$ is a classical solution of \eqref{positive pme}, we get 
\begin{equation}\label{contradiction-comparison}
{C\vep\over \hat{t}+\tau-t_\vep}\geq {\vep\over (\hat{t}+\tau-t_\vep)^2}, 
\end{equation}
which is a contradiction when $\vep$, $\tau>0$ are taken small so that $\hat{t}+\tau-t_\vep<1/C$. 
Thus $u-v$ cannot attain a positive value in $D_T$. We complete the proof of this proposition. 
\end{proof}

\begin{rem}\label{rmk comparison}
Proposition \ref{thm:3.2} can be used to prove a comparison result between a bounded nonnegative viscosity subsolution $u\in USC(\ol{D_T})$ and a positive classical solution $v\in C(\overline{D_T})\cap C^{2; 1}(D_T)$ of  \eqref{positive pme}--\eqref{positive bdry}. 
In this case, recall that the solution $v$ can be approximated by the smooth solution $v_\delta\in C^{2; 1}(D_T)$ of the following inhomogeneous Dirichlet problem with $\delta>0$ small:
\begin{numcases}{}
\partial_t u-{1\over m}\Delta u^m=0 &\quad \text{in $D_T=\Omega\times (0, T)$}\notag\\
u=\delta &\quad \text{on $\partial \Omega\times (0, T)$} \label{sigma bdry}\\
u(\cdot, 0)=v(\cdot, 0)+\delta &\quad \text{in $\Omega$}.
\end{numcases}
It is known that $v_\delta\in C^{2; 1}(D_T)$ satisfies $\inf_{D_T} v_\delta\geq \delta$ and converges to $v$ pointwise in $D_T$ as $\delta\to 0$. Consult \cite{VaBook06} and \cite{BoFRo}*{Appendix B} for more details about this approximation. Since Proposition \ref{thm:3.2} implies $u\leq v_\delta$ for any $\delta>0$, we can prove $u\leq v$ in $D_T$ by letting $\delta\to 0$. 
\end{rem}

By Theorem \ref{thm:3.1}, Proposition \ref{thm:3.2} and suitable approximations of solutions, we can now prove Theorem~\ref{thm:1.3} for the case $0<m\leq 1$, $n\geq1$. 

\begin{proof}[Proof of Theorem~{\rm\ref{thm:1.3}}] 
For $k=1,2,\dots$, let $\zeta_k$ be a $C_0^\infty({\mathbb R}^n)$ such that 
\[
0\le\zeta_k\le 1\quad\mbox{in}\quad{\mathbb R}^n,\quad
\zeta_k=1\quad\mbox{on}\quad B(0,k),\quad
\zeta_k=0\quad\mbox{outside}\quad B(0,k+1).
\]
For $i\in \{0, 1, \lambda\}$, we denote by $u_{i,k}$ the solution to the Cauchy--Dirichlet problem 
\begin{numcases}{\rm (CD)\quad }
\partial_t u-{1\over m}\Delta u^m=0 &\quad \text{in $B(0, k)\times (0, \infty)$,}\notag\\
u=0 &\quad \text{on $\partial B(0, k)\times (0, \infty)$,} \notag\\
u(\cdot, 0)=\phi_i\zeta_k &\quad \text{in $B(0, k)$}, \notag
\end{numcases}
where $B(0, R)$ denotes the open ball centered at $0$ with radius $R>0$.  
It follows from \eqref{eq:1.13} that 
\begin{equation*}
\begin{split}
M_{\alpha}(u_{1,k}(y, 0),u_{2,k}(z, 0);\lambda)
 & \le M_{\alpha}(\phi_0(y), \phi_1(z);\lambda)\\
 & \le \phi_\lambda((1-\lambda)y+\lambda z)=u_{\lambda, k+1}((1-\lambda)y+\lambda z, 0)
\end{split}
\end{equation*}
for $x$, $y\in B(0, k)$.

Let $T_{i, k}$ denote the extinction time of $u_{i, k}$ for $i\in \{0, 1, \lambda\}$ and set $Q_{\lambda, k}=B(0, k)\times (0, T_{\lambda, k+1})$.
Since $u_{0, k}$, $u_{1, k}$ are nonnegative bounded solutions of (CD), by Theorem \ref{thm:3.1} we see that their Minkowski convolution $\ol{u}_{\lambda, k}$ defined by
\[
\begin{aligned}
 & \ol{u}_{\lambda, k}(x, t)\\
 & : =\sup\big\{M_{\alpha}\left(u_{0, k}(y,t), u_{1, k}(z, t); \lambda\right)\,:\,
 y\in \overline{B(0, k)}, z\in \overline{B(0, k)}, 
\,  x=(1-\lambda)y+\lambda z\big\}
\end{aligned}
\]
for $(x, t)\in B(0, k)\times [0, \infty)$ is a viscosity subsolution with extinction time $\min\{T_{0, k}, T_{1, k}\}$. 

On the other hand, as mentioned in Remark \ref{rmk comparison}, we can approximate $u_{\lambda, k}$ by the solution $u_{\lambda, k}^\delta$ of (CD) with $\delta>0$ and boundary and initial values 
\[
u_{\lambda, k}^\delta=\delta \quad \text{on $\partial B(0, k)\times (0, \infty)$,}
\]
\[
u_{\lambda, k}^\delta(\cdot, 0)=\phi_{\lambda}\zeta_k+\delta \quad \text{in $B(0, k)$.}
\]
Since $u_{\lambda, k+1}^\delta\in C(\ol{Q}_{\lambda, k})\cap C^{2;1}(Q_{\lambda, k})$ is a positive supersolution in $Q_{\lambda, k}$ with
\[
u_{\lambda, k+1}^\delta> \ol{u}_{\lambda, k} \quad \text{on $\big(\overline{B(0, k)}\times \{0\}\big)\cup \big(\partial B(0, k)\times [0, T_{\lambda, k+1})\big)$}, 
\]
by Proposition~\ref{thm:3.2} applied to (CD),  we obtain 
\[
u_{\lambda, k+1}^\delta\geq \ol{u}_{\lambda, k} \quad \text{in $B(0, k)\times (0, T_{\lambda, k+1})$}.
\]
Passing to the limit as $\delta\to 0$, we obtain
\begin{equation}\label{approx comparison}
{u}_{\lambda,k+1}(x, t)\ge \ol{u}_{\lambda, k}(x, t)
\end{equation}
for all $x\in B(0,k)$ and $t\in (0, T_{\lambda, k+1})$, which also implies that $\ol{u}_{\lambda, k}(x, T_{\lambda, k+1})=0$ for all $x\in B(0, k)$ and therefore $\min\{T_{0, k}, T_{1, k}\}\leq T_{\lambda, k+1}$. It follows immediately that \eqref{approx comparison} holds for all $x\in B(0,k)$ and $t>0$. Hence, we get 
$$
{u}_{\lambda,k+1}((1-\lambda)y+\lambda z, t)\ge M_{\alpha}(u_{1,k}(y,t),u_{2,k}(z,t);\lambda)
$$
for all $y, z\in B(0,k)$ and $t>0$.
Letting $k\to\infty$, we have $u_{i, k}\to u_i$ locally uniformly and therefore 
\eqref{eq:1.14} holds for all $y, z\in {\mathbb R}^n$ and $t>0$. 

The above convergence of $u_{i, k}$  for $i\in \{0, 1, \lambda\}$ can be deduced as follows. Since $u_{i, k}\leq u_{i, k+1}\leq u_i$ in $B(0, k)\times (0, \infty)$ for all $k\geq 1$, we get a positive function $v_i\in C(\R^n\times (0, \infty))$ such that $u_{i, k}\to v_i$ uniformly in any compact set of $\R^n\times (0, \infty)$ as $k\to \infty$. It is not difficult to verify that $v_i$ is a very weak solution (in the sense of distribution) satisfying the initial condition $v_i=\phi_i$. Noting that by construction $v_i$ is locally Lipschitz in time and thus $\partial_t v_i\in L^1_{loc}(\R^n\times (0, \infty))$, we can apply the uniqueness result \cite{HP}*{Theorem 2.3} to show that $v_i=u_i$, which yields the locally uniform convergence of $u_{i, k}$ to $u_i$. 
\end{proof}

As a consequence of Remark \ref{rmk comparison} and a similar argument as in the proof of Theorem~\ref{thm:1.3} above, we obtain the following result. 

\begin{prop}
\label{thm:3.4}
Let $\Omega_0$, $\Omega_1$ be bounded smooth domains in $\R^n$ with $n\geq 1$. 
For $0<\lambda<1$, let $\Omega_\lambda$ be as in ~\eqref{eq:3.1}. For $i\in \{0, 1, \lambda\}$, denote $D_i=\Omega_i\times (0, \infty)$ and assume that $\phi_i\in C(\ol{\Omega_i})$ satisfies $\phi_i>0$ in $\Omega_i$. Let $0<m\leq 1$. Let $u_i\in C(\overline{D_i})$ be the nonnegative bounded solution of \eqref{eq:1.15}--\eqref{eq:1.17} with $\Omega=\Omega_i$, $\phi=\phi_i$ for $i\in \{0, 1, \lambda\}$. 
Let $\alpha:=(m-1)/2$.  If
\begin{equation}\label{eq:3.4}
\phi_\lambda((1-\lambda)x_1+\lambda x_2)\geq M_\alpha(\phi_0(x_1), \phi_1(x_2); \lambda)\quad \text{for all $x_i\in \Omega_i$, $i=1, 2$,}
\end{equation}
 then 
\begin{equation}\label{eq:3.5}
\begin{aligned}
&u_\lambda((1-\lambda)x_1+\lambda x_2, t)\\
& \ge  M_\alpha(u_0(x_0, t), u_1(x_1, t); \lambda)\quad\text{for all $x_i\in \Omega_i$, $i=1, 2$ and $t>0$.}
\end{aligned}
\end{equation}
\end{prop}

\begin{proof}
Let $T_i\in (0, \infty]$ denote the extinction times of $u_i$ for $i\in \{0, 1, \lambda\}$. 
We have seen in Theorem~\ref{thm:3.1} that $\ol{u}_{\lambda}$ defined by \eqref{eq:3.2} is a viscosity subsolution of \eqref{positive pme}--\eqref{positive bdry} with $i=\lambda$ and extinction time $\ol{T}=\min\{T_0, T_1\}$. 
It is also not difficult to show that 
\begin{equation}\label{initial-relation}
\ol{u}_\lambda(\cdot, 0)\leq \phi_\lambda=u_\lambda(\cdot, 0) \quad \text{in $\Omega_\lambda$}
\end{equation}
under the condition \eqref{eq:3.4}.  By Remark \ref{rmk comparison}, we have $\ol{u}_\lambda\leq u_\lambda$ in $\Omega_\lambda \times [0, T_\ast)$, where $T_\ast=\min\{\ol{T}, T_\lambda\}$. In particular, it implies that $\ol{T}\leq T_\lambda$ and $T_\ast=\ol{T}$.  It is clear that $\ol{u}_\lambda(\cdot, t)\equiv 0\leq u_\lambda(\cdot, t)$ for all $t>T_\ast=\ol{T}$. Hence, $\ol{u}_\lambda\leq u_\lambda$ holds in $\Omega_\lambda\times (0, \infty)$. Our proof of \eqref{eq:3.5} is complete now. 
\end{proof}

Applying Proposition \ref{thm:3.4}, in the case of $0<m\leq 1$ we recover the $\alpha$-concavity preserving property for the Cauchy--Dirichlet problem for porous medium equation with $\alpha=(m-1)/2$ established in \cite{LV2}*{Lemma~4.6}. Our proof here is different from the one in \cite{LV2}, which directly tracks the evolution of the infimum of second derivatives of $u(\cdot, t)^{\alpha}$. 

\begin{cor}
Let $\Omega$ be a bounded smooth convex domain in $\R^n$ with $n\geq 1$. 
Let $0<m\le 1$. 
Assume that $\phi\in C(\overline{\Omega})$ with $\phi=0$ on $\partial \Omega$ and $\phi>0$ in $\Omega$. 
Let $\alpha=(m-1)/2$ and $u\in C(\overline{\Omega}\times [0, \infty))$ be a solution of  \eqref{eq:1.15}--\eqref{eq:1.17}. If $\phi$ is $\alpha$-concave in $\Omega$, then so is $u(\cdot, t)$ for all $t\geq 0$.
\end{cor}
This result can be easily obtained by taking $\Omega_i=\Omega$, $\phi_i=\phi$ and $u_i=u$ for all $i\in \{0, 1, \lambda\}$ in Proposition \ref{thm:3.4}.

\subsection{Superfast diffusion in one dimension}
Before we proceed to the proof of Theorem~\ref{thm:3.8}, let us recall well-posedness results about the Cauchy problem \eqref{eq:1.8}--\eqref{eq:1.9} with $n=1$ and $-1<m\leq 0$. Based on our needs, we reorganize the main results of \cite{ERV} in the following way.

\begin{thm}\label{thm:3.6}
Let $n=1$ and $-1<m\leq 0$. Let $\phi\in L^1(\R^n)$ be such that $\phi\geq 0$, $\phi\not\equiv 0$ in $\R$. 
Then there exists a unique positive  classical solution $u\in C^{\infty}(\R\times (0, \infty))\cap C([0, \infty); L^1(\R))$ of \eqref{eq:1.8}--\eqref{eq:1.9} satisfying 
\[
\int_{\R} u(x, t)\, dx=\int_{\R} \phi\, dx\quad \text{for all $t>0$.}
\]
 Moreover, $u$ has the following properties:
\begin{itemize}
\item[(i)] for any $t>0$, $u(\cdot, t)\in L^\infty(\R)$ and  
\begin{equation}\label{eq:3.7}
-{u\over (1-m)t} \leq \partial_t u\leq {u\over (1-m)t}; 
\end{equation}
\item[(ii)] for any $T>1$, $u$ satisfies 
\begin{equation}\label{eq:3.8}
u(x, t)=O\left(|x|^{2\over m-1}\right)\,\,\, \text{as $|x|\to \infty$, uniformly for $t\in [1/T, T]$.}
\end{equation}
\end{itemize} 
In addition, in the case $-1<m<0$, if we additionally assume that $\phi\in L^1(\R)\cap L^\infty(\R) \cap C^\infty(\R)$, $\phi>0$ and $\phi^m$ is Lipschitz continuous in $\R$, then $u\in C(\R\times [0, \infty))\cap L^\infty(\R\times (0, \infty))$.
\end{thm}
These results can be found in \cite{ERV}*{Theorem~1, Proposition~2.6, Theorem~2}. 
The last statement above is an intermediate result of the first step of the proof of \cite{ERV}*{Theorem~1}. 
The boundedness of $u$ can be obtained from the estimate (2.4) in \cite{ERV}. It is worth remarking that in the proof of \cite{ERV}*{Theorem~1}, there is an important approximation to construct the solution $u$ by taking $\sigma\to 0$ for the classical solution of the Cauchy problem with initial data $u_0+\sigma$ in $\R$ with $\sigma>0$ small. We shall use this approximation to prove our comparison principle later.  

\begin{rem}\label{thm:3.7}
The inequality in \eqref{eq:3.7} enables us to obtain further results for our later use. 
For a bounded solution $u$ of  \eqref{eq:1.8}--\eqref{eq:1.9} with $n=1$ , $-1<m<0$, 
and a positive initial value $\phi\in L^1(\R)\cap L^\infty(\R) \cap C^\infty(\R)$ such that $\phi^m$ Lipschitz in $\R$,  
we can show that there exists $C>0$ such that 
\begin{equation}\label{eq:3.9}
\sup_{x\in \R}|u(x, t)-\phi(x)|\leq Ct^{1\over 2} \quad \text{for all $t>0$.}
\end{equation}
In fact, for such $\phi$,  
setting $\hat{u}(x, t)=u(x, t^2)$, by \eqref{eq:3.7}, we have 
\[
|\partial_t \hat{u}|\leq {2u\over 1-m}\quad \text{ in\,  $\R\times (0, \infty)$,}
\]
which, by the boundedness of $u$, immediately yields 
\[
\sup_{x\in \R} |\hat{u}(x, t)-\phi(x)|\leq Ct
\]
for some $C>0$. We are thus led to \eqref{eq:3.9}.
\end{rem}

We obtain the following subsolution property of the Minkowski convolution in the case $n=1$ and $-1<m<0$. 
 \begin{thm}
\label{thm:3.8}
Let $n=1$ and $-1<m<0$. Let $\phi\in L^1(\R)$ and $\phi> 0$ in $\R$. 
Suppose that for each $i=1, 2$, $u_i \in C^\infty(\R^n\times (0, \infty))\cap C([0, \infty); L^1(\R))$ is a positive bounded solution of \eqref{eq:1.8}--\eqref{eq:1.9}. 
For $\alpha=(m-1)/2$ and a fixed $\lambda\in (0, 1)$, let $u_{\lambda}$ be the $\alpha$-Minkowski convolution of $u_0, u_1$, defined by 
\eqref{eq:1.12}. 
Then $u_{\lambda}$ is a viscosity subsolution of \eqref{eq:1.8}. 
\end{thm}

\begin{proof}
Note that the operator $F_i$ stays the same as the case $0<m<1$ and therefore assumption (F) in Theorem~\ref{thm:2.2} still holds. Since $u_i$ is a solution of \eqref{eq:1.8} satisfying \eqref{eq:3.8}, by Remark \ref{thm:2.4},  we see that $\ol{u}_{\lambda}$ defined by \eqref{eq:1.12} is a positive viscosity subsolution to \eqref{eq:1.8} with $i=\lambda$. 
\end{proof}

Next, we present a comparison result in this case.

\begin{prop}\label{thm:3.9}
Let $n=1$, $-1<m< 0$ and $Q=\R \times (0, \infty)$. 
Assume that $u\in USC(\ol{Q})$ is a bounded nonnegative viscosity subsolution 
 and $v\in C^{2}(Q)\cap C(\ol{Q})$ is a bounded nonnegative solution of \eqref{eq:1.8}. Assume that $u$ satisfies \eqref{eq:3.8}. If 
\begin{equation}\label{eq:3.10}
\sup_{x\in \R}\,(u(x, t)-v(x, t))\to 0\quad \text{as $t\to 0$, }
\end{equation}
then $u\leq v$ in $\ol{Q}$.
\end{prop}
\begin{proof}
Similarly to the strategy pointed out in Remark \ref{rmk comparison}, we implement an approximation of $v$ by considering the solution $v_\delta$ of \eqref{eq:1.8} with initial value $v(\cdot, 0)+\delta$ in $\R$. It is shown in the proof of \cite{ERV}*{Theorem~1} that $v_\delta\in C^\infty(Q)$ with $v_\delta\geq \delta$ and $v_\delta\to v$ locally uniformly in $Q$ as $\delta\to 0$. 

In order to prove $u\leq v$ in $Q$, it thus suffices to show $u\leq v_\delta$ for every fixed $\sigma>0$. By \eqref{eq:3.10}, we have $u\leq v_\delta$ in $\R\times [0, s)$ for $s>0$ sufficiently small. Suppose that there exists $\hat{t}> 0$ such that $u(\cdot, \hat{t})\leq v_\delta(\cdot, \hat{t})$ in $\R$ and $\sup_\R (u(\cdot, s_j)-v(\cdot, s_j)>0$ for a sequence $s_j>\hat{t}$ with $s_j\to \hat{t}+$. 

Thanks to the decay condition \eqref{eq:3.8}, we can restrict our comparison argument in a compact set of $\overline{Q}$. Following the proof of Proposition \ref{thm:3.2}, for $\tau>0$ small, we take $\vep>0$ small such that
\[
\Phi_\vep(x, t):=u(x, t)-v_\delta(x, t)-{\vep\over \hat{t}+\tau-t}, \quad (x, t)\in \R\times [\hat{t}, \hat{t}+\tau)
\]
satisfies 
\begin{equation}\label{1d-comparison1}
0<\sup_{\R\times [\hat{t}, \hat{t}+\tau)} \Phi_\vep \leq {\vep\over \hat{t}+\tau-t}. 
\end{equation}
for $\delta>0$ small. 
Since \eqref{eq:3.8} holds for $u$, $\Phi_\vep$ attains a maximum over $\R\times [\hat{t}, \hat{t}+\tau)$ at a point $(x_\vep, t_\vep)$. The decay condition \eqref{eq:3.8} actually implies that $x_\vep$ is bounded uniformly for all such $\tau, \vep>0$. 

The rest of our proof is essentially the same as that of Proposition \ref{thm:3.2}. It follows from \eqref{1d-comparison1} that
\begin{equation}
{\vep\over \hat{t}+\tau-t_\vep}   \leq u(x_\vep, t_\vep)-v_\delta(x_\vep, t_\vep)\leq {2\vep\over \hat{t}+\tau-t_\vep}.
\end{equation}
Noticing that $\varphi(x, t)=v_\delta(x, t)+{\vep\over \hat{t}+\tau-t}$ serves as a test function for $u$ at $(x_\vep, t_\vep)$, we use the definition of viscosity subsolutions to obtain 
\[
\partial_t v_\delta(x_\vep, t_\vep)+ {\vep\over (\hat{t}+\tau-t_\vep)^2}+ F\left(u(x_\vep, t_\vep), \partial_x v_\delta (x_\vep, t_\vep), \partial_{x}^2 v_\delta(x_\vep, t_\vep)\right) \leq 0,
\]
where $F$ is given by \eqref{eq:3.3}.
Since $x_\vep$ is uniformly bounded in $\R$, we apply the uniform boundedness of $\partial_x v_\delta(x_\vep, t_\vep)$ and $\partial_x^2 v_\delta(x_\vep, t_\vep)$ to obtain Lipschitz regularity of 
\[
r\mapsto F\left(r, \partial_x v_\delta(x_\vep, t_\vep), \partial_x^2 v_\delta(x_\vep, t_\vep)\right)
\]
near $r=v_\delta(x_\vep, t_\vep)$. 
We are therefore led to 
\[
\partial_t v_\delta(x_\vep, t_\vep)+ F\left(v(x_\vep, t_\vep), \partial_x v_\delta(x_\vep, t_\vep), \partial_x^2 v_\delta(x_\vep, t_\vep)\right) 
\leq {C\vep\over \hat{t}+\tau-t_\vep}-{\vep\over (\hat{t}+\tau-t_\vep)^2}
\]
for some $C>0$. 
Adopting the supersolution property of $v_\delta$, we again end up with the contradiction \eqref{contradiction-comparison} when $\tau, \vep>0$ are taken small. It thus follows that $u\leq v_\delta$ in $Q$. 
We complete the proof by letting $\delta\to 0$. 
\end{proof}

We are now in a position to prove an analogue of Theorem~\ref{thm:1.3} in the case $n=1$ and $0<m<1$ under additional assumptions on $\phi_i$.
\begin{thm}\label{thm:3.10}
Let $n=1$, $0<m< 1$ and $0<\lambda<1$. Let $Q=\R\times (0, \infty)$. For $i\in \{0, 1, \lambda\}$, let $\phi_i\in L^1(\R)\cap L^\infty(\R) \cap C^\infty(\R)$ with $\phi_i>0$ and $\phi_i^m$ Lipschitz in $\R$. Assume that $u_i\in C^\infty(Q)\cap C(\ol{Q})$ is a positive bounded classical solution of \eqref{eq:1.8}--\eqref{eq:1.9} with $\phi=\phi_i$ for $i\in \{0, 1, \lambda\}$. Let $\alpha=(m-1)/2$. If \eqref{eq:1.13} holds for all $y, z\in \R$, then \eqref{eq:1.14} holds for all $y, z\in \R$ and $t>0$. 
\end{thm}

\begin{proof}
By Theorem~\ref{thm:3.6}, \eqref{eq:3.8} holds for all $u_i$. In view of Theorem~\ref{thm:3.8}, we see that $\ol{u}_{\lambda}$ defined by \eqref{eq:1.12} is a viscosity subsolution of \eqref{eq:1.8} 
satisfying the condition Theorem~\ref{thm:3.6}~(ii).

Let us adopt the estimates in Remark \ref{thm:3.7}. 
Utilizing \eqref{eq:3.9} for $u_0$ and $u_1$, we have
\begin{equation}\label{1dcauchy comparison}
\ol{u}_\lambda(x, t) \leq \sup\,\{M_\alpha(\phi_0(y)+Ct^{1\over 2}, \phi_1(z)+Ct^{1\over 2}; \lambda): x=(1-\lambda)y+\lambda z\},
\quad x\in{\mathbb R},
\end{equation}
for some $C>0$. 
For any $a, b\ge 0$, set
\[
g(s)=((1-\lambda)(a+s)^\alpha+\lambda (b+s)^\alpha)^{1\over \alpha}\quad\mbox{for $s\ge 0$}.
\]
Then we have, for any $s>0$, 
\begin{equation}\label{g-prime}
g'(s)=\frac{(1-\lambda) (a+s)^{\alpha-1}+\lambda (b+s)^{\alpha-1}}{\left((1-\lambda)(a+s)^\alpha+\lambda (b+s)^\alpha\right)^{\alpha-1\over \alpha}}=\left(\frac{M_{\alpha-1}(a+s, b+s; \lambda)}{M_{\alpha}(a+s, b+s; \lambda)}\right)^{\alpha-1}.
\end{equation}
We obtain $g'(s)\geq 1$ for all $s>0$ due to the property that 
\[
M_{\alpha-1}(a+s, b+s; \lambda)\leq M_{\alpha}(a+s, b+s; \lambda).
\]
On the other hand, the first expression of $g'(s)$ in \eqref{g-prime} also yields
\[
g'(s)^{1\over p}=\frac{\left((1-\lambda)^{1\over \alpha} A^p+\lambda^{1\over \alpha}B^p\right)^{1\over p}}{A+B}\leq \min\{1-\lambda, \lambda\}^{1 \over \alpha p}\frac{(A^p+B^p)^{1\over p}}{A+B}\leq \min\{1-\lambda, \lambda\}^{1 \over \alpha p},
\]
where $p={\alpha-1\over \alpha}>1$, $A=(1-\lambda)(a+s)^\alpha$ and $B=\lambda(b+s)^\alpha$. 
Hence, we have shown the following estimate: 
\begin{equation}\label{eq:3.11}
1\leq g'(s)\leq \min\{1-\lambda, \lambda\}^{1\over \alpha} \quad \text{for } s>0.
\end{equation}
In view of \eqref{1dcauchy comparison}, we therefore can find $\tilde{C}>0$ such that 
\[
\ol{u}_\lambda(x, t) \leq \sup\{M_\alpha(\phi_0(y), \phi_1(z); \lambda): x=(1-\lambda)y+\lambda z\}+\tilde{C}t^{1\over 2}
\]
for all $t>0$. The relation \eqref{eq:1.13} thus enables us to get
\[
\ol{u}_\lambda(x, t) \leq \phi_\lambda(x)+\tilde{C}t^{1\over 2}
\]
for all $(x, t)\in Q$. Applying \eqref{eq:3.9} to $u_\lambda$, we further deduce that 
\[
\sup_{x\in \R}\left(\ol{u}_\lambda(x, t) -u_\lambda(x, t)\right)\leq (C+\tilde{C})t^{1\over 2}
\]
for all $t>0$ small. 
By Proposition \ref{thm:3.9}, we obtain $\ol{u}_{\lambda}\leq u_\lambda$ in $\R^n\times (0, \infty)$, which amounts to saying that \eqref{eq:1.14} holds. Our proof is complete now.
\end{proof}


\section{Borell--Brascamp--Lieb inequality}\label{section:4}

In this section, based on the results of our preceding sections, we complete our proof of Theorem~\ref{thm:1.1}.
One of the key ingredients for this part is the large time asymptotics for  \eqref{eq:1.8}--\eqref{eq:1.9}. It is known that when $0<m<1$, $n(m-1)+2>0$ and $\phi$ satisfies
\begin{equation}\label{eq:4.1}
\phi(x)=O\left(|x|^{\frac{2}{m-1}}\right) \quad \text{as $|x|\to \infty$,}
\end{equation}
the solution $u$ satisfies 
\begin{equation}\label{eq:4.2}
t^{n\over d}|u(x, t)-U(x, t)|\to 0 \quad \text{as $t\to \infty$}
\end{equation}
uniformly for $x\in \R^n$ with $|x|\leq Ct^{1\over d}$ for any $C>0$, where $d=n(m-1)+2>0$ 
and
\[
U(x,t):=\left(\frac{2d}{1-m}\right)^{\frac{1}{1-m}}\left(\frac{t}{|x|^2+\tilde{C} t^{2\over d}}\right)^{\frac{1}{1-m}}
\]
with $\tilde{C}=c(m,n)\|\phi\|_{L^1(\R^n)}^{-\frac{2(1-m)}{d}}$ for some $c(m, n)>0$. 
It follows immediately that 
\begin{equation}\label{eq:4.3}
t^{\frac{n}{d}}u(0,t)\to C_{m,n}\|\phi\|_{L^1(\R^n)}^{\frac{2}{d}}\quad \text{as $t\to \infty$}.
\end{equation}
 This type of asymptotic behavior is well known in \cite{FKa}*{Theorem~1.1} for the case $m>1$.  
 For the fast diffusion with $0<m<1$, see such a result in \cite{VaBook}*{Theorem~18.34}.  
 In the case of superfast diffusion in one dimension, i.e., $-1<m\le 0$ and $n=1$, 
 we can obtain \eqref{eq:4.2} and \eqref{eq:4.3} without assuming \eqref{eq:4.1}; see \cite{ERV}*{Theorem~5}.

\subsection{The case of continuous integrable functions}\label{subsection:4.1}
We begin with the case when the functions are assumed to be continuous. 
\begin{prop}
\label{thm:4.1}
Let $n\geq 1$, $0<m\leq 1$ with $n(m-1)+2>0$ and $\lambda\in(0,1)$. Let $\alpha=(m-1)/2$.
Assume that $\phi_0$, $\phi_1$, $\phi_\lambda$ are nonnegative functions in $L^1(\R^n)\cap C(\R^n)$ such that \eqref{eq:1.13} 
holds for all $y, z\in \R^n$. Then 
\begin{equation}
\label{eq:4.4}
\|\phi_\lambda\|_{L^1({\mathbb R}^n)}\ge M_{\frac{\alpha}{n\alpha+1}}(\|\phi_0\|_{L^1({\mathbb R}^n)},\|\phi_1\|_{L^1({\mathbb R}^n)};\lambda).
\end{equation} 
\end{prop}

\begin{proof}
We have $n\alpha+1>0$ by the conditions $n(m-1)+2>0$ and $\alpha=(m-1)/2$. In addition, the range $0<m\leq 1$ implies $-1/2<\alpha\leq 0$. Below, we consider two different cases.  

\noindent 1) When $\alpha=0$ (equivalently $m=1$), for $i\in\{0, 1, \lambda\}$ and any $\ell>0$, we consider the solution $U_{i,\ell}$ of \eqref{eq:1.8}--\eqref{eq:1.9} with $\phi=\Phi_{i,\ell}$, where
\begin{equation}\label{eq:4.5}
\Phi_{i,\ell}(x):=\min\{\phi_i(x),\ell\}, \quad x\in{\mathbb R}^n.
\end{equation}
Since $\phi_0$, $\phi_1$, $\phi_\lambda\in L^1(\R^n)\cap C(\R^n)$ satisfy \eqref{eq:1.13} for all $y, z\in \R^n$, it is not difficult to see that $\Phi_{0, \ell}, \Phi_{1, \ell}, \Phi_{\lambda, \ell}\in L^1(\R^n)\cap C(\R^n)$ are bounded in $\R^n$ and also satisfy \eqref{eq:1.13} for all $y, z\in \R^n$. Then, by Theorem~\ref{thm:1.3} we have 
\[
U_{\lambda,\ell}((1-\lambda)y+\lambda z,t)\ge M_{0}(U_{1,\ell}(y,t), U_{2,\ell}(z,t);\lambda)
\]
for all $y, z\in{\mathbb R}^n$, and $t>0$. In particular, letting $y=z=0$ we have 
\begin{equation}
\label{eq:4.6}
U_{\lambda,\ell}(0,t)\ge M_{0}(U_{1,\ell}(0,t), U_{2,\ell}(0,t);\lambda).
\end{equation} 
Since for any $x\in \R^n$
$$
U_{i,\ell}(x,t)=(4\pi t)^{-\frac{n}{2}}\int_{{\mathbb R}^n}\exp\left(-\frac{|x-y|^2}{4t}\right)\Phi_{i,\ell}(y)\,dy,
$$
Lebesgue's dominated convergence theorem implies that,  as $t\to \infty$, 
$$
t^{\frac{n}{2}}U_{i,\ell}(0,t)\to (4\pi)^{-\frac{n}{2}}\|\Phi_{i,\ell}\|_{L^1({\mathbb R}^n)}.
$$
This together with \eqref{eq:4.6} yields 
$$
\|\Phi_{\lambda,\ell}\|_{L^1({\mathbb R}^n)}\ge M_0(\|\Phi_{1,\ell}\|_{L^1({\mathbb R}^n)},\|\Phi_{2,\ell}\|_{L^1({\mathbb R}^n)};\lambda). 
$$
Letting $\ell\to\infty$, we obtain inequality~\eqref{eq:4.4} in the case of $\alpha=0$. 

\noindent 2) If $-1/2< \alpha<0$ (equivalently $0<m<1$), for $i\in \{0, 1, \lambda\}$, we set
\[
{\Phi}_{i,\ell}(x):=\min\left\{\phi_i(x),\ell(1+|x|)^{1\over \alpha}\right\}.
\]
Let $U_{i,\ell}$ still denote the solution to problem \eqref{eq:1.8}--\eqref{eq:1.9} with $\phi={\Phi}_{i,\ell}$. 
Then, by Theorem~\ref{thm:1.3} we have 
\[
{U}_{\lambda,\ell}((1-\lambda)y+\lambda z, t)\ge M_{\alpha}({U}_{1,\ell}(y, t),{U}_{2,\ell}(z, t);\lambda)
\]
for all $x$, $y\in{\mathbb R}^n$, and $t>0$, which in particular yields, 
\begin{equation}\label{eq:4.7}
U_{\lambda,\ell}(0, t)\ge M_{\alpha}(U_{1,\ell}(0, t), U_{2,\ell}(0, t);\lambda)\quad \text{for all $t>0$.}
\end{equation}
On the other hand, 
since 
$$
\Phi_{i,\ell}(x)=O(|x|^{1\over \alpha})\quad\mbox{as}\quad |x|\to\infty,
$$
by using \eqref{eq:4.3}, we get
$$
t^{\frac{n}{d}}U_{i,\ell}(0,t)\to C_{m,n}\|\Phi_{i,\ell}\|_{L^1(\R^n)}^{\frac{2}{d}}\quad \text{as $t\to \infty$}.
$$
This together with \eqref{eq:4.7} implies that
$$
\|\Phi_{\lambda,\ell}\|_{L^1(\R^n)}^{\frac{2}{d}}
\ge M_{\alpha}\left(\|\Phi_{1,\ell}\|_{L^1(\R^n)}^{\frac{2}{d}},|\Phi_{2,\ell}\|_{L^1(\R^n)}^{\frac{2}{d}};\lambda\right).
$$
Noticing that
$$
\frac{2\alpha}{d}
=\frac{2\alpha}{n(m-1)+2}
=\frac{\alpha}{n\alpha+1},
$$
 we have
$$
\|\phi_\lambda\|_{L^1(\R^n)}\ge \|\Phi_{\lambda,\ell}\|_{L^1(\R^n)}
\ge M_{\frac{\alpha}{n\alpha+1}}\left(\|\Phi_{1,\ell}\|_{L^1(\R^n)},\|\Phi_{2,\ell}\|_{L^1(\R^n)};\lambda\right)
$$
for all $\ell>0$. Therefore, letting $\ell\to\infty$, we obtain inequality~\eqref{eq:4.4} and conclude the proof for the case $n\alpha+1>0$.  
Thus Proposition~\ref{thm:4.1} follows.
\end{proof}


\begin{prop}
\label{thm:4.2}
Let $n=1$, $-1<m< 0$, $\alpha=(m-1)/2$ and $0<\lambda<1$. 
For $i\in \{0, 1, \lambda\}$, let $\phi_i\in L^1(\R)\cap L^\infty(\R) \cap C^\infty(\R)$ with $\phi_i>0$ and $\phi_i^m$ Lipschitz in $\R$. 
Assume that $\phi_i$ satisfy \eqref{eq:1.13} for all $y, z\in \R$ and $i\in\{0, 1,\lambda\}$. Then \eqref{eq:4.4} holds with $n=1$. 
\end{prop}

\begin{proof}
Since the large time behavior \eqref{eq:4.3} for the solution of \eqref{eq:1.8}--\eqref{eq:1.9} holds without needing \eqref{eq:4.1}, we can apply the same argument as in the proof of Proposition \ref{thm:4.2}. Indeed, let $u_i$ be the solution of \eqref{eq:1.8}--\eqref{eq:1.9} (in the sense of Theorem~\ref{thm:3.6}) with $\phi=\phi_i$, by Theorem~\ref{thm:3.10}, we have 
\[
u_{\lambda}((1-\lambda)y+\lambda z, t)
\ge  M_\alpha(u_{0}(y, t), u_{1}(z, t); \lambda)\quad \text{for all $y, z\in \R$, $t>0$}
\]
and therefore 
\[
u_{\lambda}(0, t)
\ge  M_\alpha(u_{0}(0, t), u_{1}(0, t); \lambda) \quad\text{for all $t>0$.}
\]
Adopting \eqref{eq:4.3}, we get \eqref{eq:4.4} immediately.
\end{proof}

The borderline case $\alpha=-1/2$ in one dimension, corresponding to the choice $m=0$,  can be proved by approximation. It can also be obtained from the cases $\alpha<-1/2$ by the homogeneity argument in Appendix~\ref{section:A}. The result for the case of $\alpha=-1$ and $m=-1$ also follows by approximation, as clarified later.


\subsection{The case of $L^1$ functions}\label{subsection:4.2}

We have shown Theorem~\ref{thm:1.1} when the functions satisfy certain additional regularity assumptions. 
This section is devoted to the proof of BBL without smoothness beyond $L^1$ regularity by applying further regularization. 

We use the standard mollification to regularize $\phi_i$. Recall that in $\R^n$, the mollifier with $\vep>0$ is defined by 
\begin{equation}\label{eq:4.8}
\rho_\varepsilon(x):=\varepsilon^{-n}\rho(\varepsilon^{-1}x)\quad\mbox{for}\quad x\in{\mathbb R}^n,
\end{equation}
where $\rho$ is a $C_0({\mathbb R}^n)$-function given by
\[
\rho(x)=\begin{cases}
\displaystyle{c_n\exp\left(-\frac{1}{1-|x|^2}\right)} &\text{for $x\in B(0,1)$,}\\
0 &\text{for $x\in{\mathbb R}^n\setminus B(0,1)$}.
\end{cases}
\]
Here $c_n$ is a positive constant to be chosen such that $\|\rho\|_{L^1({\mathbb R}^n)}=1$.

\begin{proof}[Proof of Theorem~{\rm\ref{thm:1.1}}]
The following proof applies to the choice 
$f=\phi_0$, $g=\phi_1$ and $h=\phi_\lambda$
in Theorem~\ref{thm:1.1}. 
Let us split our discussion into several different cases. 

\medskip 

\noindent {\bf Case 1:} 
Consider the case of $n\geq 1$, $-1/2<\alpha< 0$ and $n\alpha+1>0$.
Let $\alpha=(m-1)/2$ as before. For any $c\geq 1$ and $f\in L^1(\R^n)$, we first define $J_{\delta}^c[f]$  by 
\begin{equation}\label{eq:4.9}
J_{\delta}^c[f](x)=\min\left\{f(x)+c\delta, \delta^{\beta}(1+|x|)^{1\over \alpha}\right\}, \quad x\in \R^n, 
\end{equation}
where we choose $\beta<0$ close to $0$ so that 
\begin{equation}\label{eq:4.10}
1+n\alpha(1-\beta)>0.
\end{equation} 
It is easily seen that $J_{\delta}^c[f]\in L^\infty(\R^n)$ with $J_{\delta}^c[f]>0$ in $\R^n$. Also, due to the condition $n\alpha+1>0$, we see that $J_{\delta}^c[f]\in L^1(\R^n)$. 

Moreover, when $\delta>0$ is small, we have 
$J_{\delta}^c[f](x)=\delta^\beta(1+|x|)^{1\over \alpha}$ for all $x\in \R$ satisfying $|x|\geq R_\delta$, where 
\begin{equation}\label{eq:4.11}
R_\delta:=\delta^{\alpha(1-\beta)}-1>0.
\end{equation}
Then, we get
\[
\|J_{\delta}^c[f]-f\|_{L^1(\R)}\leq \int_{|x|\leq R_\delta} |J_{\delta}^c[f](x)-f(x)|\, dx+ \int_{|x|\geq R_\delta} |J_{\delta}^c[f](x)-f(x)|\, dx.
\]
%
Let us fix arbitrarily $x\in \R^n$ satisfying $|x|\leq R_\delta$ and $f(x)<\delta^\beta(1+|x|)^{1\over \alpha}$. 
We then immediately get 
\begin{equation}\label{eq:4.12}
0\leq  J_\delta^c[f](x)-f(x)\leq (f(x)+c\delta)-f(x)= c\delta
\end{equation}
and therefore $|J^c_\delta[f](x)-f(x)|\leq c\delta$ for such $x$. It follows that 
\begin{equation}\label{eq:4.13}
\int_{|x|\leq R_\delta} |J^c_{\delta}[f](x)-f(x)|\, dx \leq C_1                
           \delta R_\delta^n +\int_{|x|\leq R_\delta} \left(f(x)-\delta^\beta(1+|x|)^{1\over \alpha}\right)_+\,dx
\end{equation}
for some $C_1>0$. 
We have $\delta R_\delta^n\to 0$ as $\delta\to 0$, thanks to \eqref{eq:4.10} and  \eqref{eq:4.11}. 
The second term on the right hand side above also converges to $0$ by Lebesgue's dominated convergence theorem. 
Indeed, for any fixed $R>0$, we have as $\delta\to 0$
\[
F_\delta(x):=\left((f(x)-\delta^\beta(1+|x|)^{1\over \alpha}\right)_+\to 0
\]
pointwise for all $x\leq R$, which by Lebesgue's dominated convergence theorem yields  
\[
\int_{|x|\leq R} F_\delta(x)\, dx\to 0.
\]
Since
\[
\begin{aligned}
\int_{|x|\leq R_\delta} F_\delta(x)\,dx
&=\int_{|x|\leq R} F_\delta(x)\, dx+ \int_{R\leq |x|\leq R_\delta} F_\delta(x)\,dx\\
&\leq \int_{|x|\leq R} F_\delta(x)\, dx+\int_{|x|\geq R} f(x)\, dx,
\end{aligned}
\]
It follows that 
\[
\limsup_{\delta\to 0}\int_{|x|\leq R_\delta} F_\delta(x)\,dx\leq \int_{|x|\geq R} f(x)\, dx,
\]
By sending $R\to \infty$, we immediately obtain 
\[
\int_{|x|\leq R_\delta} F_\delta(x)\,dx\to 0\quad\mbox{as $\delta\to 0$}.
\]  
As a result, from \eqref{eq:4.13} we are led to 
\begin{equation}\label{eq:4.14}
\int_{|x|\leq R_\delta} \left|J_{\delta}^c[f](x)-f(x)\right|\, dx\to 0\quad \text{as $\delta\to 0$}. 
\end{equation}
 On the other hand, it is not difficult to see that 
\begin{equation}\label{eq:4.15}
\begin{aligned}
\int_{|x|\geq R_\delta} |J_{\delta}^c[f](x)-f(x)|\, dx &\leq \int_{|x|\geq R_\delta} J_{\delta}^c[f](x)\, dx+ \int_{|x|\geq R_\delta} f(x)\, dx\\
&\leq C_2\delta^{1+n\alpha(1-\beta)} +\int_{|x|\geq R_\delta} f\, dx
\end{aligned}
\end{equation}
for some $C_2>0$, which yields 
\begin{equation}\label{eq:4.16}
\int_{|x|\geq R_\delta}\left|J_{\delta}^c[f](x)-f(x)\right|\, dx\to 0\quad \text{as $\delta\to 0$. }
\end{equation}
Combining \eqref{eq:4.14} and \eqref{eq:4.16}, we obtain 
\begin{equation}\label{eq:4.17}
\left\|J_{\delta}^c[f](x)-f(x)\right\|_{L^1(\R^n)}\to 0\quad \text{as $\delta\to 0$.}
\end{equation}

Noticing that  \eqref{eq:1.13} holds for almost every $y, z\in \R$, by \eqref{eq:3.11} we can take $a\geq 1$ such that, for any $\delta>0$, it holds
\begin{equation}\label{eq:4.18}
M_\alpha({\phi}_{0}(y)+\delta, {\phi}_{1}(z)+\delta ; \lambda)\leq M_\alpha(\phi_0(y), \phi_1(z); \lambda)
+a\delta\leq \phi_\lambda((1-\lambda)y+\lambda z)+a\delta. 
\end{equation}
We fix such $a\geq 1$. In view of the $\alpha$-concavity of the function $(1+|x|)^{1\over \alpha}$, 
the relation \eqref{eq:4.18} together with \eqref{eq:4.9} yields 
\begin{equation}\label{eq:4.19}
M_\alpha\left(J_\delta^1[\phi_0](y), J_\delta^1[\phi_1](z); \lambda\right)\leq J_\delta^a[\phi_\lambda]((1-\lambda)y+\lambda z)
\end{equation}
for almost all $y, z\in \R^n$. 

We next adopt the mollifier \eqref{eq:4.8} for further regularization.   For any $\varepsilon>0$, set 
\begin{equation}\label{eq:4.20}
\tilde{\phi}_{i}=\left(J_\delta^1 [\phi_i]^\alpha \ast \rho_\vep\right)^{1\over \alpha} 
\end{equation} 
for $i= 1, 2$ and 
\begin{equation}\label{eq:4.21}
\tilde{\phi}_{\lambda}=\left(J_\delta^a [\phi_\lambda]^\alpha \ast \rho_\vep\right)^{1\over \alpha}.
\end{equation} 
%
Note that $\tilde{\phi}_{i}$ constructed in \eqref{eq:4.20} and \eqref{eq:4.21} are positive, continuous in ${\mathbb R}^n$. 
Moreover, letting  $\varepsilon\to 0$ and then $\delta\to 0$,  
we obtain $\tilde{\phi_i}\to \phi_i$ in $L^1(\R^n)$ for all $i\in \{0, 1, \lambda\}$. 

 Using the mollification, from \eqref{eq:4.18} we obtain 
\begin{equation}\label{eq:4.22}
M_\alpha\left(\tilde{\phi}_{1}(y), \tilde{\phi}_{2}(z); \lambda\right)\leq \tilde{\phi}_{\lambda}((1-\lambda) y+\lambda z)
\end{equation}
for all $y, z\in \R^n$. In view of Proposition \ref{thm:4.1} we get 
\begin{equation}\label{eq:4.23}
\|\tilde{\phi}_\lambda\|_{L^1({\mathbb R}^n)}\ge M_{\frac{\alpha}{n\alpha+1}}(\|\tilde{\phi}_1\|_{L^1({\mathbb R}^n)},\|\tilde{\phi}_2\|_{L^1({\mathbb R}^n)};\lambda).
\end{equation}
Passing to the limit as $\varepsilon\to 0$, $\delta\to 0$ we end up with the desired inequality for $\phi_i$. 

\medskip 

\noindent {\bf Case 2:} 
Consider the case $n=1$ and $-1<\alpha<-1/2$. 
Our approximation is actually the same as in the previous case. Let us define $J_\delta^c$ for $c\geq 1$ in the same manner as in \eqref{eq:4.9} with $n=1$ and $\beta\in ((\alpha+1)/\alpha, 0)$ such that \eqref{eq:4.10} still holds. Take $\tilde{\phi}_i$ as in \eqref{eq:4.20} and \eqref{eq:4.21}. Hence, we still get \eqref{eq:4.22} for all $y, z\in \R$ and $\tilde{\phi}_i\to \phi_i$ in $L^1(\R)$ as $\vep, \delta\to 0$ for $i\in \{0, 1, \lambda\}$. 

The only difference for the current case is that, in order to use Proposition \ref{thm:4.2}, we need our approximation $\tilde{\phi}_i$ to be in $L^1(\R)\cap L^\infty(\R) \cap C^\infty(\R)$ with $\tilde{\phi}_i>0$ and $(\tilde{\phi}_i)^m$ Lipschitz in $\R$. These properties of $\tilde{\phi}_i$ do hold. Indeed, by our construction it is clear that $\tilde{\phi}_i>0$ in $\R$ and $\tilde{\phi}_i\in L^\infty(\R)\cap C^\infty(\R)$. Moreover, since $x\mapsto J_{\delta}^c[f](x)^\alpha$ is linear in $(-\infty, R_\delta]$ and $[R_\delta, \infty)$ for any $f\in L^1(\R)$ and $c\geq 1$, where $R_\delta>0$ is given by \eqref{eq:4.11}, we have 
\[
\tilde{\phi}_{i}(x)=J_{\delta}^c[\phi_i](x) =\delta^\beta(1+|x|)^{1\over \alpha }\quad \text{for all $|x|>R_\delta+1$}
\]
when $\vep>0$ is small. This implies that $\tilde{\phi}_i\in L^1(\R)$ and $(\tilde{\phi}_i)^m$ is Lipschitz continuous in $\R$ for $i\in \{0, 1, \lambda\}$. 

Hence, applying Proposition \ref{thm:4.2} to $\tilde{\phi}_i$, we are led to \eqref{eq:4.23} again. 
Sending $\vep\to 0$ and then $\delta\to 0$, we finish the proof of BBL in the case of $n=1$ and $-1<\alpha<-1/2$. This also implies the case $\alpha=-1/2$ by the argument in Appendix \ref{section:A}.

\medskip

\noindent {\bf Case 3:} Consider the case of $n\geq 1$, $\alpha=0$.  This case corresponds to PL. 
It seems that $J_\delta^c$ does not meet our needs in this case. We introduce a different operator $K_{\delta, \gamma}^c$ with $\delta>0$, $\gamma\in (0, 1]$ and $c\geq 1$.  For any $f\in L^1(\R^n)$, let
\[
K_{\delta, \gamma}^c[f](x)=\min\left\{\max\{f(x), c\delta^\gamma\},\ \delta^{\beta}e^{-|x|}\right\}, \quad x\in \R^n,
\]
where we fix $-\gamma<\beta<0$. It is clear that $K_{\delta, \gamma}^c[f]\in L^1(\R^n)$.

Although $K_{\delta, \gamma}^c$ appears slightly different from $J_\delta^c$, we can still prove $K_{\delta, \gamma}^c[f]\to f$ in $L^1$ as $\delta\to 0$ by similar arguments for \eqref{eq:4.17} in Case~1. In fact, 
choosing
$R_\delta:=(\beta-\gamma)\log \delta$ enables us to obtain estimates analogous to \eqref{eq:4.13} and \eqref{eq:4.15}: for $\delta>0$ sufficiently small, we have
\[
\int_{|x|\leq R_\delta} |K^c_{\delta, \gamma}[f](x)-f(x)|\, dx \leq C_1                
           \delta^\gamma R_\delta^n +\int_{|x|\leq R_\delta} \left(f(x)-\delta^\beta e^{-|x|}\right)_+\,dx,
\]
\[
\begin{aligned}
\int_{|x|\geq R_\delta} & |K^c_{\delta, \gamma}[f](x)-f(x)|\, dx \leq \int_{|x|\geq R_\delta} \delta^\beta e^{-|x|}\, dx+ \int_{|x|\geq R_\delta} f(x)\, dx\\
&\leq \delta^\beta C_2 e^{-{R_\delta\over 2}}+\int_{|x|\geq R_\delta} f\, dx =C_2 \sqrt{c}\delta^{\beta+\gamma\over 2}+\int_{|x|\geq R_\delta} f\, dx,
\end{aligned}
\]
for some constants $C_1, C_2>0$. Since our choice of $\beta$ yields $\beta+\gamma>0$, the $L^1$-convergence of $K_{\delta, \gamma}^c[f]$ to $f$ as $\delta\to 0$ then follows immediately. 

Fix $\ell>0$ arbitrarily large and let $\Phi_{i, \ell}$ be given by \eqref{eq:4.5}. Since \eqref{eq:1.13} holds for $\phi_i$ with $i\in \{0, 1, \lambda\}$, we deduce that
\[
\Phi_{0, \ell}(y)^{1-\lambda}\Phi_{1, \ell}(z)^\lambda\leq \Phi_{\lambda, \ell}((1-\lambda)y+\lambda z)
\]
for almost all $y, z\in \R^n$. It follows that 
\begin{equation}\label{eq:4.24}
\max\{\Phi_{0, \ell}(y), \delta\}^{1-\lambda}\max\{\Phi_{1, \ell}(z), \delta\}^\lambda\leq \max\{\Phi_{\lambda, \ell}((1-\lambda)y+\lambda z),\  c_\ell\delta^{\lambda_0}\}
\end{equation}
for almost all $y, z\in \R^n$ and $\delta>0$ small, where $\lambda_0=\min\{1-\lambda, \lambda\}$ and $c_\ell=\max\{\ell^\lambda, \ell^{1-\lambda}\}$.

We extend our mollification in \eqref{eq:4.20} and \eqref{eq:4.21} to the case $\alpha=0$. For $i=0, 1$, we adopt $K_{\delta, \gamma}^c$ with $\gamma=c=1$ and set
\[
\tilde{\Phi}_{i, \ell}= \exp \left(\log \left(K_{\delta, 1}^1[\Phi_{i, \ell}]\right) \ast \rho_\vep\right).
\]
On the other hand, for $\Phi_{\lambda, \ell}$ we use 
$K_{\delta, \gamma}^c$ with $\gamma=\lambda_0$, $b=c_\ell $
and take 
\[
\tilde{\Phi}_{\lambda, \ell}= \exp \left(\log \left(K_{\delta, \lambda_0}^{c_\ell}[\Phi_{\lambda, \ell}]\right) \ast \rho_\vep\right).
\]
By the log-concavity of the function $e^{-|x|}$, the relation \eqref{eq:4.24} implies that
\[
(1-\lambda)\log \left(K_{\delta, 1}^1[\Phi_{0, \ell}]\right)(y)+ \lambda \log \left(K_{\delta, 1}^1[\Phi_{1, \ell}]\right)(z)\leq \log \left(K_{\delta, \lambda_0}^{c_\ell}[\Phi_{\lambda, \ell}]\right) \left((1-\lambda)y+\lambda z\right)
\]
for almost all $y, z\in \R^n$. It then follows that 
\[
M_0\left(\tilde{\Phi}_{0, \ell}(y),  \tilde{\Phi}_{1, \ell}(z); \lambda\right)\leq \tilde{\Phi}_{\lambda, \ell}((1-\lambda)y+\lambda z)
\]
holds for all $y, z\in \R^n$. Since $\tilde{\Phi}_{i, \ell}\in L^1(\R^n)\cap C(\R^n)$, we can apply Proposition \ref{thm:4.1} to obtain
\[
\|\tilde{\Phi}_{\lambda, \ell}\|_{L^1(\R^n)}\geq M_0\left(\|\tilde{\Phi}_{0, \ell}\|_{L^1(\R^n)}, \|\tilde{\Phi}_{1, \ell}\|_{L^1(\R^n)}\right).
\]
Letting $\vep\to 0$ and then $\delta\to 0$, by the $L^1$-convergence established above we get
\[
\|\Phi_{\lambda, \ell}\|_{L^1(\R^n)}\geq M_0\left(\|\Phi_{0, \ell}\|_{L^1(\R^n)}, \|\Phi_{1, \ell}\|_{L^1(\R^n)}\right).
\]
Since $\Phi_{i, \ell}\to \phi_i$ in $L^1(\R^n)$ as $\ell\to \infty$, we pass to the limit to complete our proof of PL for $\phi_i$ with $i\in \{0, 1, \lambda\}$.

\medskip

\noindent {\bf Case 4:} Consider the case of $n\geq1$, $\alpha=-1/n$. 
This case is not included in the discussion above. However, we can get it through approximations by first having the inequality with exponent 
\[
\alpha=-{1\over n}+\vep \quad (\vep>0)
\]
for functions $\hat{\phi}_i$ ($i\in \{0, 1, \lambda\}$) given by
\[
\hat{\phi}_i(x)=\min\left\{\phi_i(x)^{1\over 1-n\vep},\  {1\over \vep}\right\}, \quad x\in \R^n
\]
and then sending $\vep\to 0+$. Therefore the proof of Theorem~\ref{thm:1.1} is complete. 
\end{proof}

\section{Equality condition for Pr\'ekopa--Leindler inequality}\label{section:5}

In this section we continue to adopt our PDE approach to further explore the equality condition. 
When the functions in question are assumed to be continuous and compactly supported, we are able to recover the equality condition for PL, thanks to special fine properties of the heat equation, including eventual log-concavity and backward uniqueness. 
The case of more general BBL  
turns out to be a more challenging problem and related comments are given at the end of the section.  

\begin{thm}\label{thm:5.1}
Let $n\geq 1$ and $\lambda\in (0, 1)$. 
Suppose that $\phi_0, \phi_1, \phi_\lambda$ are nonnegative continuous functions in ${\mathbb R}^n$ with compact support.
Assume that \eqref{eq:1.13} holds with $\alpha=0$, that is, 
\begin{equation}\label{eq:5.1}
\phi_\lambda((1-\lambda)y+\lambda z)\geq \phi_0(y)^{1-\lambda} \phi_1(z)^\lambda
\end{equation}
for all $y, z\in \R^n$. Assume also that
\begin{equation}\label{eq:5.2}
\int_{\R^n} \phi_\lambda(x)\, dx=\left(\int_{\R^n} \phi_0(x)\, dx\right)^{1-\lambda} \left(\int_{\R^n} \phi_1(x)\, dx\right)^\lambda. 
\end{equation}
Then, $\phi_0, \phi_1, \phi_\lambda$ are log-concave in $\R^n$ such that 
\begin{equation}\label{eq:5.3}
\phi_0(x)= k\phi_1 \left(x+\eta\right)=k^\lambda \phi_\lambda(x+\lambda \eta),
\quad x\in{\mathbb R}^n,
\end{equation}
for some $\eta\in \R^n$ and $k>0$.
\end{thm}

We need the following lemma in our proof of Theorem~\ref{thm:5.1}.

\begin{lem}\label{thm:5.2}
Assume that $W_1, W_2\in C^2(\R^n)$ are strongly convex functions in $\R^n$. 
 Assume in addition that for any $\xi\in \R^n$, there exist $x, y\in \R^n$ such that $\xi=\nabla W_1(x)=\nabla W_2(y)$ and $\nabla^2 W_1(x)=\nabla^2 W_2(y)$ hold. Then there exist $\eta\in \R^n$ and $b\in \R$ such that 
$$
W_1(x)=W_2(x+\eta)+b\quad \text{ for all $x\in \R^n$.}
$$
\end{lem}
\begin{proof}
The key tool of our proof is the Legendre transform (convex conjugate) $W^*$ for a strongly convex function $W\in C^2(\R^n)$ in ${\mathbb R}^n$, that is,
$$
W^\ast(\xi)=\sup_{x\in \R^n}\{\langle x, \xi\rangle -W(x)\}, \quad \xi\in \R^n.
$$
One can also replace the supremum above by a maximum. In fact, for each $\xi\in \R^n$, the strong convexity of $W$ enable us to find a unique $x\in \Omega$ where $\nabla W(x)=\xi$. 
We can further show $W^\ast\in C^2(\R^n)$ and obtain the following properties (see for instance \cite{HLBook}*{Corollary~4.2.10}): 
\begin{align}
\label{eq:5.4}
 & W(x)=\sup_{\xi\in \R^n}\{\langle x, \xi\rangle -W^\ast(\xi)\},\\
\label{eq:5.5}
 & \nabla^2 W^\ast(\xi) \nabla^2 W(x)=I\quad \text{with}\quad \xi=\nabla W(x),
\end{align}
for all $x\in{\mathbb R}^n$.
Let $W_i^\ast$ denote the conjugate of $W_i$ for $i=1, 2$. By assumptions, for any $\xi\in \R^n$, we have $x_1$, $x_2\in \R^n$ such that 
\[
\xi=\nabla W_1(x_1)=\nabla W_2(x_2), \quad \nabla^2 W_1(x_1)=\nabla^2 W_2(x_2).
\]
In view of the relation \eqref{eq:5.5}, we are then led to $\nabla^2 W_1^\ast(\xi)=\nabla^2 W_2^\ast(\xi)$ for all $\xi\in \R^n$. This immediately yields the existence of $\eta\in \R^n$ and $b\in \R$ such that
\[
W_1^\ast(\xi)=W_2^\ast(\xi)-\langle \eta, \xi\rangle -b\quad \text{for all $\xi\in \R^n$.}
\]
We then can use the duality \eqref{eq:5.4} to deduce that, for all $x\in \R^n$,  
\begin{equation}\label{eq:5.6}
W_1(x)=\sup_{\xi \in \R^n} \{\langle x+\eta, \xi \rangle -W_2^\ast(\xi)+b\}=W_2(x+\eta)+b
\end{equation}
as desired. The proof is thus complete.
\end{proof}

Let us now prove Theorem~\ref{thm:5.1}.

\begin{proof}[Proof of Theorem~{\rm\ref{thm:5.1}}]
Let $Q=\R^n\times (0, \infty)$ and $u_i\in C^\infty(Q)\cap L^\infty(Q)\cap C(\ol{Q})$ be the unique solution to the heat equation with initial data $\phi=\phi_i$ for $i\in \{0, 1, \lambda\}$. Take the Minkowski convolution $\ol{u}_\lambda$ of $u_0, u_1$ as in \eqref{eq:1.12}. Applying Theorem~\ref{thm:1.3}, we get $\ol{u}_\lambda\leq u_\lambda$ in $Q$. We can also use Proposition \ref{thm:4.1} to show that, for any $t>0$, 
\[
\int_{\R^n} \ol{u}_\lambda(x, t)\, dx\geq \left(\int_{\R^n} u_0(x, t)\, dx\right)^{1-\lambda} \left(\int_{\R^n} u_1(x, t)\, dx\right)^\lambda 
\]
and therefore 
\begin{equation}\label{eq:5.7}
\int_{\R^n} u_\lambda(x, t)\, dx\geq \int_{\R^n} \ol{u}_\lambda(x, t)\, dx\geq \left(\int_{\R^n} u_0(x, t)\, dx\right)^{1-\lambda} \left(\int_{\R^n} u_1(x, t)\, dx\right)^\lambda. 
\end{equation}
Note that the heat equation preserves the total mass of the solution $u_i$, that is,  
\[
\int_{\R^n} u_i(x, t)\, dx=\int_{\R^n} \phi_i(x)\, dx
\]
holds for all $t>0$ and $i\in \{0, 1, \lambda\}$. It thus follows from \eqref{eq:5.2} that
\[
\int_{\R^n} u_\lambda(x, t)\, dx= \left(\int_{\R^n} u_0(x, t)\, dx\right)^{1-\lambda} \left(\int_{\R^n} u_1(x, t)\, dx\right)^\lambda. 
\]
Hence, from \eqref{eq:5.7} we see that $u_\lambda=\ol{u}_\lambda$ in $Q$.

It is known that solutions of the heat equation with nonnegative and compactly supported initial data enjoy the so-called eventual log-concavity: there exist $t_\ast>0$ and $\sigma>0$ such that $\nabla^2 v_i(\cdot, t_\ast)< -\sigma I$ with $v_i=\log u_i$ for all $i\in \{0, 1, \lambda\}$. See \cite{LV1}*{Theorem~5.1}, \cite{IshSaTa1}*{Theorem~3.3} and more recent work \cite{Ish-pre23} for more details about this geometric property of the heat equation.

For such $t_\ast>0$ and any $x_\ast\in \R^n$, it is easily seen from the relation $u_\lambda=\ol{u}_\lambda$ that 
\[
(x, t)\mapsto \min\left\{ 
(1-\lambda)v_0(y, t)+\lambda v_1(z, t):\,  y, z\in \R^n, 
\,  x =(1-\lambda)y+  \lambda z
\right\} -v_\lambda(x, t)
\]
attains a maximum at $(x_\ast, t_\ast)$. Then there exist $x_1, x_2\in \R^n$ such that $(1-\lambda)x_1+\lambda x_2=x_\ast$ and 
$$
(y, z, t)\mapsto (1-\lambda)v_0(y, t)+\lambda v_1(z, t)-v_\lambda((1-\lambda)y+\lambda z, t)
$$
attains a minimum at $(x_1, x_2, t_\ast)\in \R^{2n}$. 
We thus obtain the following relations: 
$$
(1-\lambda)v_1(x_1, t_\ast)+\lambda v_2(x_2, t_\ast)=v_\lambda (x_\ast, t_\ast).
$$
$$
(1-\lambda)  \partial_t v_1(x_1, t_\ast) + \lambda \partial_t v_2(x_2, t_\ast)=\partial_t v_\lambda(x_\ast, t_\ast),
$$
\begin{equation}\label{eq:5.8}
 \nabla v_1(x_1, t_\ast)= \nabla v_2(x_2, t_\ast)=\nabla v_\lambda(x_\ast, t_\ast),
\end{equation}
and
\begin{equation}\label{eq:5.9}
\left((1-\lambda) X_1^{-1}+\lambda X_2^{-1}\right)^{-1}\leq Y,
\end{equation}
where we denote $X_1=\nabla^2 v_1 (x_1, t_\ast)$, $X_1=\nabla^2 v_1(x_2, t_\ast)$ and $Y=\nabla^2 v_\lambda (x_\ast, t_\ast)$.

The proof of \eqref{eq:5.9} can be found in \cite{ALL}*{Lemma~4}. 
In fact, it is easily observed that, for any fixed $h, h_1, h_2\in \R^n$ such that $h=(1-\lambda) h_1+\lambda h_2$, 
\[
v_\lambda(x_\ast+rh, t_\ast)\geq  (1-\lambda) v_1(x_\ast+rh_1, t_\ast)+\lambda 
v_2(x_1+rh_2, t_\ast)
\]
holds for all $r>0$ small, which by Taylor expansion and \eqref{eq:5.8} yields, 
\begin{equation}\label{eq:5.10}
\langle Y h, h\rangle\geq (1-\lambda)\langle X_1 h_1, h_1\rangle+\lambda \langle X_2 h_2, h_2\rangle.
\end{equation}
It turns out that, under the constraint $h=(1-\lambda) h_1+\lambda h_2$, 
\[
h_i=X_i^{-1}((1-\lambda) X_1^{-1}+\lambda X_2^{-1})^{-1} h, \quad i=1, 2
\]
minimizes the right hand side of \eqref{eq:5.10}.  
Therefore \eqref{eq:5.9} follows immediately.  Moreover, we have $X_1, X_2, Y<0$ because of the log-concavity of $u_i$.

We next apply the equation for $v_i$ to obtain
\[
\partial_t v_i(x_i, t_\ast)-|\nabla v_i(x_i, t_\ast)|^2- \Delta v_i(x_i, t_\ast)=0
\]
for $i=1, 2$ and 
\[
\partial_t v_\lambda(x_\ast, t_\ast)-|\nabla v_\lambda(x_\ast, t_\ast)|^2-\Delta v_\lambda(x_\ast, t_\ast)=0.
\]
By \eqref{eq:5.9}, it follows that 
\begin{equation}\label{eq:5.11}
(1-\lambda) \tr X_1+\lambda \tr X_2 = \tr Y\geq \tr \left((1-\lambda) X_1^{-1}+\lambda X_2^{-1}\right)^{-1}.
\end{equation}
It is known \cite{CoSa2}*{Lemma~2} that for $X_1, X_2<0$, one actually has 
\[
(1-\lambda) \tr X_1+\lambda \tr X_2 \leq \tr \left((1-\lambda) X_1^{-1}+\lambda X_2^{-1}\right)^{-1},
\]
and the equality holds only when $X_1=X_2$. Hence, by \eqref{eq:5.11} we obtain $X_1=X_2=Y$, 
which is equivalent to 
$$
\nabla^2 v_1(x_1, t_\ast)=\nabla^2 v_2(x_2, t_\ast)=\nabla^2 v_\lambda(x_\ast, t_\ast).
$$

Applying further Lemma \ref{thm:5.2} to $W_i=-v_i$, we obtain $b\in \R$, $c>0$ and $\eta\in \R^n$ such that $cx_1+\eta=x_2$ and 
\[
{1\over c^2}v_2(cx+\eta, t_\ast)+b=v_1(x, t_\ast)
\]
holds for all $x\in \R^n$. Differentiating the latter relation  and letting $x=x_1$, we have 
\[
\nabla v_1(x_1, t_\ast)= {1\over c}\nabla v_2(x_2, t_\ast),
\]
which, combined with \eqref{eq:5.8} yields $c=1$. We therefore deduce that 
\begin{equation}\label{eq:5.12}
v_1(x, t_\ast)=v_2(x+\eta, t_\ast)+b. 
\end{equation}
Analogously, we can show that 
\begin{equation}\label{eq:5.13}
v_1(x, t_\ast)=v_\lambda (x+\lambda \eta, t_\ast)+\lambda b.
\end{equation}
Combining \eqref{eq:5.12} and \eqref{eq:5.13}, we thus obtain \eqref{eq:1.20}
for all $x\in \R^n$ with $k=e^b$. 
 
Since $u_0, ku_1$ and $k^\lambda u_\lambda$ are all solutions to the heat equation, 
by the backward uniqueness for the heat equation \cite{EKPV}*{Theorem~1}, we deduce the desired equality condition \eqref{eq:5.3}. 

The log-concavity of $\phi_i$ in $\R^n$ for all $i\in \{0, 1, \lambda\}$ can be easily obtained from \eqref{eq:5.1} and~\eqref{eq:5.3}. In fact, plugging \eqref{eq:5.3} into \eqref{eq:5.1}, we have 
\[
\phi_\lambda((1-\lambda)y+\lambda z)\geq \phi_0(y+\lambda \eta)^{1-\lambda}\phi_1(z+(\lambda -1)\eta)^\lambda
\]
for all $y, z\in \R^n$. Hence, for any $y', z'\in \R^n$, by letting $y=y'-\lambda \eta$ and $z=z'+(1-\lambda)\eta$, we have 
\[
\phi_\lambda((1-\lambda)y'+\lambda z')\geq \phi_0(y')^{1-\lambda}\phi_1(z')^\lambda,
\]
which verifies the log-concavity of $\phi_\lambda$. The log-concavity of $\phi_0$ and $\phi_1$ also holds thanks to \eqref{eq:5.3} again. The proof is complete.
\end{proof}

\begin{rem}For the equality condition of the case with $\alpha\in [-1/n, 0)$, one may repeat the same argument as in the case of $\alpha=0$ to formally derive the equality condition~\eqref{eq:1.5}. However, we are not able to complete the proof rigorously, 
due to several inherent difficulties including the absence of results on the backward uniqueness for the fast diffusion equation \eqref{eq:1.8}.
\end{rem}

\appendix

\section{Relation between BBL with different exponents}\label{section:A}

In this appendix we show how to obtain BBL for some exponent $\alpha=q$, including its equality conditions, from another BBL with exponent $\alpha=p$ with $p\leq q$. In particular, our result shows that PL implies BBL for all $q>0$ and that BBL with $\alpha={-1/n}$ implies every BBL. 
The argument is based on homogeneity.  
We include a proof here for the sake of completeness and for the benefit of the reader.
\begin{prop}\label{thm:A.1}
Let $-1/n\leq p\leq q\leq+\infty$. 
Assume that the statement of Theorem~{\rm\ref{thm:1.1}} holds with $\alpha=p$. 
Then it also holds for $\alpha=q$.
Moreover, the equality condition 
implies the equality condition
in BBL with $\alpha=q$; 
in other words, if the statement of Theorem~{\rm \ref{thm:1.2}} holds with $\alpha=p$, then it holds for $\alpha=q$ too.
\end{prop}
\begin{proof}
Assume that BBL 
holds for some $\alpha=p\geq-1/n$. 
Let $f$, $g$, $h$ be nonnegative integrable functions in ${\mathbb R}^n$ and set
\begin{align*}
 &  F=\int_{\R^n}  f(x)\,dx, \qquad 
G=\int_{\R^n}  g(x)\,dx, \quad H=\int_{\R^n}  h(x)\,dx\, \\
 & p'=\left\{\begin{array}{ll}
p/(1+np)\,&\text{if }p>-1/n,\vspace{3pt}\\
-\infty\,&\text{if }p=-1/n.
\end{array}\right.
\end{align*}
%
%
Let $q>p$ and assume that
\begin{equation}\label{eq:A.1}
 h(1-\lambda)y+\lambda z)\geq M_q( f(y), g(z);\lambda)\,
\end{equation}
for almost all $y$, $z\in \R^n$. 
We prove that
\begin{equation}\label{eq:A.2}
H\geq M_{q'}(F,G;\lambda)\quad\text{with }\ q'=q/(1+nq). 
\end{equation}
Set $M=(1-\lambda)F^{q'}+\lambda G^{q'}$. Then $M_{q'}(F,G;\lambda)=M^{1/q'}=M^{(1+nq)/q}$ and 
$$
\begin{aligned}
\dfrac{ h(1-\lambda)y+\lambda z)}{M_{q'}(F,G;\lambda)}&\geq  \left[(1-\lambda)\dfrac{ f(y)^q}{M^{1+nq}}+\lambda\dfrac{ g(z)^q}{M^{1+nq}}\right]^{1/q}\\
&=\left[\dfrac{1-\lambda}{M}\left(\dfrac{ f(y)}{M^n}\right)^q+\dfrac{\lambda}{M}\left(\dfrac{ g(z)}{M^n}\right)^q\right]^{1/q}\\
&=\left[\dfrac{(1-\lambda)F^{q'}}{M}\left(\dfrac{ f(y)}{M^nF^{1/(1+nq)}}\right)^q+\dfrac{\lambda G^{q'}}{M}\left(\dfrac{ g(z)}{M^nG^{1/(1+nq)}}\right)^q\right]^{1/q}.
\end{aligned}
$$
for almost all $y$, $z\in \R^n$.
Now set 
$$
\hat{h}=\frac{h}{M_{q'}(F,G;\lambda)},\quad
\tilde{f}=\frac{f}{M^nF^{1/(1+nq)}},\quad
\tilde{g}=\frac{g}{M^nG^{1/(1+nq)}},\quad
\mu=\dfrac{\lambda G^{q'}}{M}.
$$
Since
$$
1-\mu=\frac{(1-\lambda)F^{q'}+\lambda G^{q'}}{M}-\dfrac{\lambda G^{q'}}{M}
=\frac{(1-\lambda)F^{q'}}{M},
$$ 
the above inequality can be written as
\begin{equation}\label{eq:A.3}
\hat{h}((1-\lambda)y+\lambda z)\geq M_q(\tilde{f}(y),\tilde{g}(z);\mu). 
\end{equation}
Notice that in the argument of $\hat{h}$ we have $\lambda$ as weight, while on the right hand side we have a $q$-mean with weight $\mu$.
On the other hand, we have
$$
(1-\lambda)y+\lambda z=(1-\mu)\frac{M}{F^{q'}}y+\mu\frac{M}{G^{q'}}z=(1-\mu)\eta+\mu\zeta, 
$$
where
$$
\eta=\frac{M}{F^{q'}}\,y, \qquad\zeta=\frac{M}{G^{q'}}\,z. 
$$
Hence, further setting 
$$
\hat{f}(\eta)=\tilde{f}\left(\frac{F^{q'}}{M}\eta\right)=\frac{ f\left(\frac{F^{q'}}{M}\eta\right)}{M^nF^{1/(1+nq)}}, \qquad\hat{g}(\zeta)=\tilde{g}\left(\frac{G^{q'}}{M}\zeta\right)=\frac{ g\left(\frac{G^{q'}}{M}\zeta\right)}{M^nG^{1/(1+nq)}}, 
$$
we can rewrite \eqref{eq:A.3} as
$$
\hat{h}((1-\mu)\eta+\mu\zeta)\geq M_q(\hat{ f}(\eta),\hat{g}(\zeta);\mu).
$$
By the relation $q>p$ and the monotonicity of power means,  we have 
$$
\hat{h}((1-\mu)\eta+\mu\zeta)\geq M_q(\hat{f}(\eta),\hat{g}(\zeta);\mu)\geq M_p(\hat{f}(\eta),\hat{g}(\zeta);\mu). 
$$
Then the assumption of BBL with exponent $p$ is satisfied by 
$\hat{f}$, $\hat{g}$ and $\hat{h}$,
whence
\begin{equation}\label{eq:A.4}
\int_{\R^n}\hat{h}(\xi)\,d\xi
\geq M_{p'}\left(\int_{\R^n}\hat{ f}(\eta)\,d\eta,\int_{\R^n}\hat{g}(\zeta)\,d\zeta;\ \mu\right). 
\end{equation}
Since
\begin{align*}
 & \int_{\R^n}\hat{h}(\xi)\,d\xi=\frac{H}{M_{q'}(F,G;\lambda)}, \\
 & \int_{\R^n}\hat{f}(\eta)d\eta
=\frac{1}{M^nF^{1/(1+nq)}}\left(\frac{M}{F^{q'}}\right)^n\int_{\R^n} f(y)dy=1, \\
& \int_{\R^n}\hat{g}(\zeta)d\zeta=\frac{1}{M^n G^{1/(1+nq)}}\left(\frac{M}{G^{q'}}\right)^n\int_{\R^n} g(y)dy=1, 
\end{align*}
it follows from \eqref{eq:A.4} that
\begin{equation}\label{eq:A.5}
\frac{H}{M_{q'}(F,G;\lambda)}\geq 1, 
\end{equation}
which coincides with the desired conclusion \eqref{eq:A.2}.

Regarding equality conditions, note that the equality in BBL for $\alpha=q$ yields equality in \eqref{eq:A.5}, which in turn coincides with equality in \eqref{eq:A.4}, that is, equality in BBL when $\alpha=p$. Then if Theorem~\ref{thm:1.2} is known to hold for $\alpha=p$, we obtain that $f,\,g$ and $h$ coincides, up to suitable homotheties. To complete the proof we need to prove that $h$ is $q$-concave, but this is now a straightforward consequence of \eqref{eq:A.1}.
\end{proof}

\end{document}